\documentclass[10pt]{elsarticle}

\usepackage{charter} 
\usepackage{etex}
\usepackage{ifpdf}
\usepackage{amsmath} 
\usepackage{enumerate}
\usepackage{graphicx,subfigure,epstopdf,epsfig}
\usepackage[hang,flushmargin]{footmisc} 
\usepackage{sansmath}
\usepackage{color,soul}
\sethlcolor{yellow}
\usepackage{hyperref}
\usepackage{amssymb}                  
\usepackage[all]{xy}
\usepackage[width=16cm,height=24cm,top=2cm]{geometry}

\newtheorem{thm}{Theorem}

\newtheorem{prop}{Proposition}
\newtheorem{lem}{Lemma}

\newtheorem{as}{Assumption}
\usepackage{algorithm,algpseudocode} 
\algnewcommand\algorithmicinput{\textbf{Input:}}
\algnewcommand\Input{\item[\algorithmicinput]}
\algnewcommand\algorithmicoutput{\textbf{Output:}}
\algnewcommand\Output{\item[\algorithmicoutput]}
\newcommand{\norm}[1]{\left\lVert#1\right\rVert}
\newenvironment{proof}{\paragraph{Proof:}}{\hfill$\square$}

\DeclareMathOperator{\spn}{span}

\newcommand{\vertiii}[1]{{\left\vert\kern-0.25ex\left\vert\kern-0.25ex\left\vert #1 
    \right\vert\kern-0.25ex\right\vert\kern-0.25ex\right\vert}}
    
\usepackage{graphicx,subfigure,epstopdf,epsfig}
\usepackage[font=small,labelfont=bf]{caption}
\usepackage{booktabs,siunitx}
\usepackage{multirow}

\newcounter{defcounter}
\newenvironment{myequation}{
\addtocounter{equation}{-1}
\refstepcounter{defcounter}

\begin{equation}}
{\end{equation}}

\begin{document}

\begin{frontmatter}
\title{Algorithms and analyses for stochastic optimization for turbofan noise reduction using parallel reduced-order modeling}
\author{Huanhuan Yang}
\ead{hyang3@fsu.edu}
\author{Max Gunzburger}
\ead{mgunzburger@fsu.edu}
\address{Department of Scientific Computing, Florida State University, Tallahassee, FL 32306, USA}

\begin{abstract}
Simulation-based optimization of acoustic liner design in a turbofan engine nacelle for noise reduction purposes can dramatically reduce the cost and time needed for experimental designs. Because uncertainties are inevitable in the design process, a stochastic optimization algorithm is posed based on the conditional value-at-risk measure so that an ideal acoustic liner impedance is determined that is robust in the presence of uncertainties. A parallel reduced-order modeling framework is developed that  dramatically improves the computational
efficiency of the stochastic optimization solver for a realistic nacelle geometry. The reduced stochastic optimization solver takes less than 500 seconds to execute. In addition, well-posedness and finite element error analyses of the state system and optimization problem are provided. 
\end{abstract}

\begin{keyword}
stochastic Helmholtz equation \sep conditional value at risk \sep proper orthogonal decomposition  \sep turbofan noise reduction
\end{keyword}

\end{frontmatter}

\section{Introduction}
Aircraft noise is a major constraint on expanding and improving the air transport environment throughout the world. With the popularization of air transportation, aviation noise mitigation has always been an interesting topic for researchers and engineers \cite{Vanker2009, Biswas2013}. Noise emission at take-off and landing from the high-bypass turbofan, the only choice of engine for commercial aircrafts because of its lower fuel consumption \cite{Azimi2014}, is mainly contributed by the engine fan noise \cite{Kempton2011}. 
At take-off, the fan rotational speed is supersonic and this makes the noise (known as ``buzz-saw'' noise) propagate upstream the inlet \cite{McAlpine2007419}. During landing, the fan speed is low and the noise is caused by the interaction of the blades with the inlet flow.

The fan noise radiation can be effectively damped by the equipment of an optimally designed acoustic liner in the engine nacelle. To this end, one needs to address some design challenges including but not limited to the choice of acoustic liner material and layer structure. The  performance of acoustic liners can be evaluated in experiments by means of ground tests \cite{Schuster} or in dedicated experimental test rigs \cite{Ferrante}. Simulation-based optimization on the liner design, however, can dramatically reduce the experimental cost and time. In particular, simulations have been performed in \cite{Cao2007Estim} for the search of liner impedance factors, by solving an optimization problem towards minimizing fan noise radiation. In this paper, we still focus on the estimation of optimal liner impedance factors.

Mathematical models governed by partial differential equations (PDEs) often contain coefficients (or boundary condition data), such as the acoustic wavenumber in the Helmholtz equation for sound propagation, that are not exactly known due to incomplete knowledge or an inherent variability in the system. These uncertainties should be introduced into the model by treating the parameters as random variables. Optimization of the resulting stochastic system would be more complex than the deterministic one, but its accommodation to model uncertainties provides a more robust and realistic tool for practical application. In this paper, we take into account uncertainties on the acoustic wavenumber due to variability in the weather, and on the fan noise source due to incomplete knowledge. We formulate the optimization on the conditional value-at-risk (CVaR) measure \cite{Rockafellar00optimizationof}, which quantifies the conditional expectation of the sound energy provided that the sound is above a certain threshold. The optimization based on CVaR measure is expected to determine optimal impedance factor that are robust to uncertainty. Solving the stochastic optimization problem would facilitate the optimal acoustic liner design with different significance levels.

PDE constrained Optimization integrated with uncertainty quantification, although more reliable in application, is computationally formidable due to the inclusion of stochastic variables and therefore the dramatically
 increased number of realizations of deterministic PDEs. It is natural to apply a reduced-order modeling technique with the aim of dramatically reducing the computational cost of each realization. In this work, we apply the Proper Orthogonal Decomposition (POD) approach to the reduction of the Helmholtz model. The POD method has been used in many different fields such as fluid-structure interaction \cite{BertagnaVeneziani} and electrophysiology \cite{HHpodDEIM2016}, and has been applied to the Helmholtz equation for a different purpose \cite{Volkwein2011compHelm}.
POD is effective because of its optimal ability to approximate the snapshots with minimized error. To the best of our knowledge, it is also the most efficient method applying to the non-coercive Helmholtz equation, since it has no need of computing the inf-sup constant that is needed by the greedy reduced basis method. In fact, greedy reduced basis method applying to non-coercive elliptic problems is more involved than coercive problems \cite{SugataThesis2007}.

In the work, we build a parallel reduced-order modeling framework for the stochastic optimization to reduce the turbofan noise radiation. The simulations for fan noise propagation are performed on realistic geometry. Whereas the computation of the CVaR measure of the full-order Helmholtz solutions is forbidding, the reduced stochastic optimization problem can be solved within 500 seconds. Numerical experiments based on minimizing the CVaR measure indicate: with 95\% certainty the acoustic noise energy can be optimally controlled within 48.66\% of the noise level associated with the hard-wall condition without acoustic liner.

The paper is outlined as follows. We formulate a stochastic optimization problem in Section~\ref{sec:SPDEcontrol} that is based on the CVaR measure described in Section~\ref{CVaR-intr}. We also introduce notions about proper orthogonal decomposition (POD)-based reduced-order modeling (Section \ref{pod}) and then, in Section \ref{numScheme}, apply it together with a BFGS optimization strategy, to discrete acoustic liner optimization problem. Numerical illustrations of the accuracy and efficiency gains enabled by using the reduced-order model in a parallel processing environment are provided in Section~\ref{sec:resutls}. At last, mathematical and numerical analyses of the state equation and the optimization problem are provided in Section \ref{analysis}.


\section{The stochastic optimization problem}\label{sec:SPDEcontrol}

The engine inlet part of the turbofan nacelle is depicted in Figure~\ref{engineGeo} which shows the noise radiation streaming out of the inlet from the fan. 
A three-dimensional acoustic mesh of the fan intake (the domain $D \in \mathbb{R}^3$) is shown in Figure~\ref{mesh3D} and is used here for acoustic simulation. 
Its boundary surface is composed of five parts: the fan noise source boundary is denoted by $\Gamma_1$; the area of the attached acoustic liner material is denoted by $\Gamma_2$; $\Gamma_3$ denotes the near-field boundary whereas $\Gamma_4$ denotes the boundary far from the noise source, assuming that the Sommerfeld radiation boundary condition holds (as done in \cite{Cao2007Estim}); the acoustic wave propagation is assumed to be axisymmetric with the symmetry plane denoted by $\Gamma_5$.

\begin{figure}[h!]
\begin{center}
\includegraphics[scale=0.4]{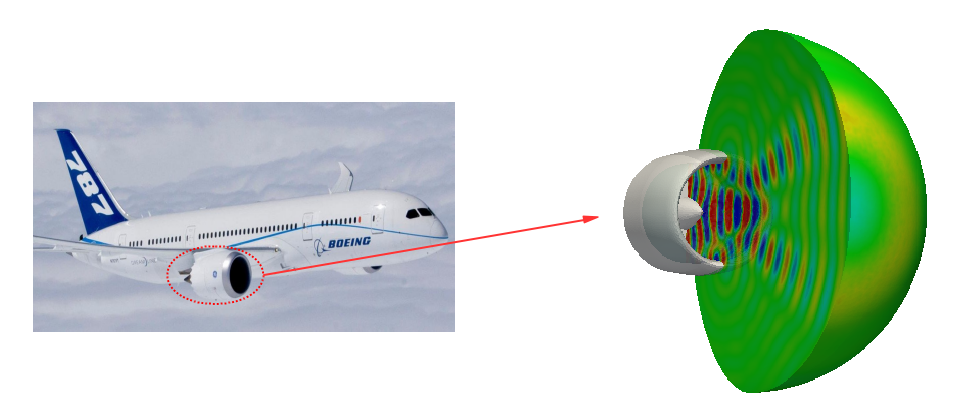}
\caption[Caption for LOF]{The turbofan engine inlet \protect\footnotemark with a typical simulation example showing noise radiation streaming out the inlet from the fan.}
\label{engineGeo}
\end{center}
\end{figure}
\footnotetext{Picture source of the aircraft: http://bestwallpaperhd.com/wp-content/uploads/2014/02/Boeing-Aircraft.jpg}

\begin{figure}[h!]
\begin{center}
\includegraphics[scale=0.4]{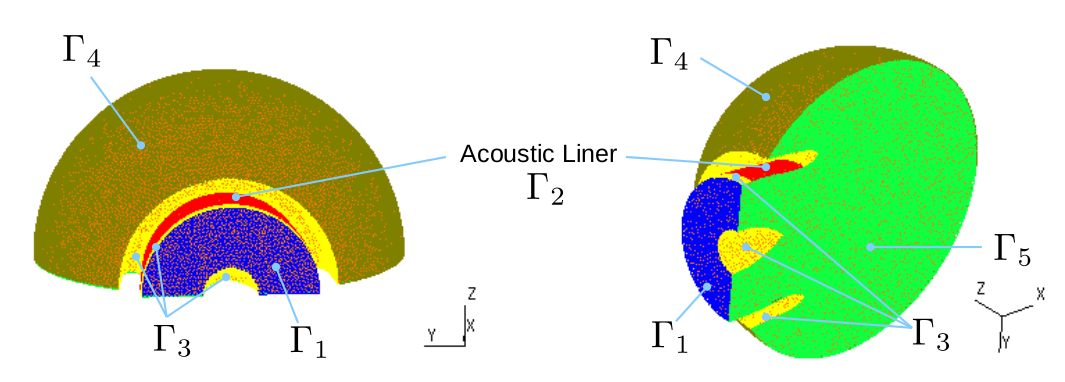}
\caption{Acoustic mesh of fan intake. The boundary surface constitutes five parts which are color-labelled as follows: fan noise source boundary $\Gamma_1$ in blue, acoustic liner boundary $\Gamma_2$ in red, near field boundary  $\Gamma_3$ in yellow, far field boundary $\Gamma_4$ in olive, and the symmetry plan boundary $\Gamma_5$ in green.}
\label{mesh3D}
\end{center}
\end{figure}

Aircraft noise propagation is governed by the Helmholtz equation $-\Delta p(\mathbf{x}) - k^2 p(\mathbf{x}) = 0$ with appropriate boundary conditions, where $p(\mathbf{x})$ denotes the complex-valued acoustic pressure at $\mathbf{x}=(x,y,z)\in D$. Here, the dependent variable $p(\mathbf{x})$ is appropriately non-dimensionalized \cite{SugataThesis2007}.
If $f$ denotes the frequency, $\omega = 2\pi f$ the angular frequency, and $c$ the sound speed, then $k = {\omega}/{c} > 0$ is the acoustic wavenumber. The wavenumber $k$ measures the amount of phase change of the sound waveform per meter. The speed of sound obviously changes with temperature and air humidity, and negligibly with atmospheric pressure and sound frequency \cite{Martin2011}. To take into consideration these uncertainties in weather conditions, we treat accordingly the wavenumber $k$ as a random variable. A Dirichlet boundary condition is imposed on $\Gamma_1$ to model the fan noise source: $p(\mathbf{x})|_{\Gamma_1} = \mu g_{\Gamma_1}(\mathbf{x})$, where $g_{\Gamma_1}(\mathbf{x})$ is prescribed whereas $\mu=\mu_{\rm r}+i\mu_{\rm i}$ is a random complex variable used to represent the variability of fan noise amplitude.

Formally, let $(\Omega, \mathcal{F} , P )$ denote a complete probability space, where $\Omega$, $\mathcal{F}$ and $P$ are the set of outcomes $\omega \in \Omega$, the $\sigma$-algebra collecting events, and the probability measure, respectively.  The measure $P : \mathcal{F} \to [0, 1]$ with $P (\Omega) = 1$ assigns probability to the events. 
We assume the random vector $\vartheta =[k,\mu_{\rm r},\mu_{\rm i}]: \Omega \to \Lambda=\Lambda_1\times\Lambda_2\times\Lambda_3$ is endowed with the joint probability density function $\rho: \Lambda\to [0, +\infty]$, and $\Lambda$ is bounded in $\mathbb{R}^+\times\mathbb{R}^2$. The resulting stochastic Helmholtz equation subject to boundary conditions is given by
\begin{equation}\label{PDEstate}
\left\{
\begin{array}{ll}
-\Delta p(\mathbf{x},\omega) - k(\omega)^2 p(\mathbf{x},\omega) = 0 &  \mbox{in } D \\[0.2cm]
p(\mathbf{x},\omega) = \mu(\omega)g_{\Gamma_1}(\mathbf{x}) &  \mbox{on } \Gamma_1 \\[0.2cm]
\frac{\partial p(\mathbf{x},\omega)}{\partial \mathbf{n}} + i\frac{k(\omega)}{\xi}p(\mathbf{x},\omega) = 0 & \mbox{on } \Gamma_2 \\[0.2cm]
\frac{\partial p(\mathbf{x},\omega)}{\partial \mathbf{n}} = 0 & \mbox{on }\Gamma_3\cup\Gamma_5  \\[0.2cm]
\frac{\partial p(\mathbf{x},\omega)}{\partial \mathbf{n}} + ik(\omega)p(\mathbf{x},\omega) = 0 & \mbox{on } \Gamma_4 .
\end{array}
\right.
\end{equation}
Here, $\mathbf{n}$ denotes the outward normal direction. The parameter $\xi=\xi_{\rm r}+i \xi_{\rm i} \in \mathbb{C}$ is the impedance factor of the acoustic liner whose real part $\xi_{\rm r}$, which should be positive for physical reasons, represents resistance and the imaginary part $\xi_{\rm i}$  reactance. For the random variable $p(\mathbf{x},\omega)$, we do not distinguish from the notation $p(\mathbf{x},\vartheta)$ in the sequel.
 

We define the function spaces
$V_0 = \{\phi\in H^1(D; \mathbb{C}): \phi|_{\Gamma 1} = 0 \}$, 
$V_{\vartheta} = \{p\in H^1(D; \mathbb{C}): p|_{\Gamma 1} = \mu g_{\Gamma_1} \}$,  
$Y_0 = L^2_{\rho}(\Lambda; V_0)$,
and $Y = \{p(\cdot,\vartheta): \Lambda \to V_{\vartheta},  \int_{\Lambda} \Vert p(\cdot,\vartheta)\Vert_{V_{\vartheta}}^2\rho(\vartheta)d\vartheta < \infty \}$. 
Then, a weak formulation of (\ref{PDEstate}) is given as follows: find $p \in Y$ such that 
\begin{equation}\label{weakHelmState}
\int_{\Lambda}a_{\vartheta}(p,\phi)\rho(\vartheta)d\vartheta = 0 \quad \forall \phi\in Y_0,
\end{equation}
where
\begin{equation*}
a_{\vartheta}(p,\phi) = \int_D\nabla p\cdot\nabla \overline{\phi} d\mathbf{x} - k^2\int_D p\overline{\phi} d\mathbf{x} + \frac{ik}{\xi}\int_{\Gamma_2}p\overline{\phi} ds + ik\int_{\Gamma_4}p\overline{\phi} ds
\end{equation*}
with $\overline{(\cdot)}$ denoting the complex conjugate.

The optimization problem we consider is to determine the impedance factor $\xi$ of the acoustic liner that
minimizes the amount of noise propagated from the engine inlet. For the uncertainties involved in the Helmholtz equation governing the noise radiation, we need an appropriate measure $\sigma$ to characterize the probability distribution of the aircraft noise. Mathematically, the optimization problem we consider is given by
\begin{myequation}
\min\limits_\xi \Big\{ \frac{1}{2}\sigma\Big[\frac{1}{\gamma_{\rm p}}\int_D|p(\mathbf{x},\cdot; \xi)|^2d\mathbf{x} \Big] +\frac{\gamma}{2}|\xi|^2\Big\},
\end{myequation} 
where $p(\mathbf{x},\vartheta)$ satisfies the weak equation (\ref{weakHelmState}). The minimization is over the physical domain of $\xi$ which is omitted for simplicity. The constant $\gamma_{\rm p}$ is chosen to scale the energy of the acoustic potential and $\gamma$ is the regularization coefficient. In this paper, we consider the operator $\sigma$ to be a risk measure, a concept that is popular in science or finance to control large deviations or tail probabilities. In particular, we take the conditional value-at-risk (CVaR) measure  that was first developed in the finance community \cite{Rockafellar00optimizationof} and later applied to PDE-constrained optimization \cite{Kouri2016cvar}. The CVaR measure is described in Section~\ref{CVaR-intr}.

Solutions of the Helmholtz equation, in both the deterministic or stochastic cases, are generally expensive to  obtain. The Helmholtz equation is non-coercive so that iterative finite element solutions converges much slower than for coercive problems.
Moreover, a stochastic solution could require a large number of realizations of deterministic solutions.
Motivated by this, in this paper, we focus on the reduced-order modeling of the Helmholtz equation in a parallel framework, and apply it for the stochastic optimization problem mentioned above. In particular, we employ the classical proper orthogonal decomposition reduced-modeling approach that is discussed in Section~\ref{pod}.

\subsection{The conditional value-at-risk measure}\label{CVaR-intr}

Let $X(\vartheta)$ denote a general cost function with uncertainties denoted by the random vector $\vartheta: \Omega \to \Lambda$.
The underlying probability distribution of $\vartheta$ is assumed to have density $\rho(\vartheta)$. The distribution function of $X$ is 
$$\Psi(\alpha) = \int_{\{\vartheta: X(\vartheta) \leq \alpha \}} \rho(\vartheta) d\vartheta.$$ 
At a specified confidence level $\beta \in (0, 1)$, 
the corresponding value-at-risk ($\mbox{VaR}_{\beta}$) is defined as the $\beta$-quantile of $X$, that is,
$$\text{VaR}_{\beta}[X] = \min\{\alpha\in \mathbb{R}: \Psi(\alpha) \geq \beta \}.$$ 
It is the lowest $\alpha$ such that, with probability $\beta$, the value of $X$ will not exceed $\alpha$. 
Using this concept, at probability level $\beta$, the {\em conditional value-at-risk} $\text{CVaR}_{\beta}$ is defined as the conditional expectation
$$
\text{CVaR}_{\beta}[X] = \mathbb{E}\big[X \big | X\geq\text{VaR}_{\beta}[X] \big] = \frac{1}{1-\beta}\int_{\{\vartheta: X(\vartheta) \geq \text{VaR}_{\beta}[X]\}} X(\vartheta) \rho(\vartheta) d\vartheta.
$$
$\mbox{CVaR}_{\beta}$ measures the conditional mean value of the cost above the amount $\text{VaR}_{\beta}[X]$. The second equality results from the probability $P\big[ X \geq \text{VaR}_{\beta}[X] \big] =  1- \beta$. 

%
%
%

Based on the $\mbox{CVaR}_{\beta}$ concept, the stochastic optimization problem we propose is to determine
\begin{myequation}
\min\limits_{\xi}\Big\{\frac{1}{2}\text{CVaR}_\beta\Big[\frac{1}{\gamma_{\rm p}}\int_D|p(\mathbf{x},\cdot; \xi)|^2d\mathbf{x}\Big] +\frac{\gamma}{2}|\xi|^2\Big\}. \label{Jbeta}
\end{myequation}
In \cite{Rockafellar00optimizationof}, it is proved that $\text{CVaR}_{\beta}[X]$ can be characterized in terms of 
$$\text{CVaR}_{\beta}[X] = \min\limits_{\alpha\in\mathbb{R}}F_{\beta}(\alpha; X),$$
where 
$$
F_{\beta}(\alpha; X) = \alpha + \frac{1}{1-\beta}\int_\Lambda\big[X(\vartheta)-\alpha \big]^+\rho(\vartheta)d\vartheta
$$
with $[x]^+=\max\{x,0\}$. It is also shown in \cite[Theorem 2]{Rockafellar00optimizationof} that problem (\ref{Jbeta}) is equivalent to
\begin{myequation}\label{Jbetaalpha}
\min\limits_{\xi,\alpha}\Big\{\frac{1}{2}\bigg[ \alpha + \frac{1}{1-\beta}\int_\Lambda\Big[\frac{1}{\gamma_{\rm p}}\int_D|p(\mathbf{x},\vartheta; \xi)|^2d\mathbf{x} - \alpha \Big]^+\rho(\vartheta)d\vartheta \bigg] +\frac{\gamma}{2}|\xi|^2\Big\}. 
\end{myequation}

To solve (\ref{Jbetaalpha}), non-smooth optimization can be avoided by smoothing the plus function $[ x ]^+$ appearing in the CVaR measure.  For this purpose, we use the smoothed plus function \cite{Kouri2016cvar} 
\begin{equation*}
h_\varepsilon(x) = \left\{
\begin{array}{ll}
0	&	\text{if	} x\leq -\frac{\varepsilon}{2} \\[0.1cm]
\frac{(x+\varepsilon/2)^3}{\varepsilon^2	} - \frac{(x+\varepsilon/2)^4}{2\varepsilon^3} & \text{if	  }  -\frac{\varepsilon}{2} < x < \frac{\varepsilon}{2} \\[0.1cm]
x	& \text{if	 } x\geq \frac{\varepsilon}{2}.
\end{array}
\right.
\end{equation*}
in $C^2(\mathbb{R})$.
The resulting smoothed problem is
\begin{myequation}\label{Jbetaalpha-eps}
\min\limits_{\xi,\alpha}
\Big\{\frac{1}{2}\bigg[ \alpha + \frac{1}{1-\beta}\int_\Lambda h_\varepsilon\Big(\frac{1}{\gamma_{\rm p}}\int_D|p(\mathbf{x},\vartheta; \xi)|^2d\mathbf{x}-\alpha \Big)\rho(\vartheta)d\vartheta \bigg] +\frac{\gamma}{2}|\xi|^2\Big\}. 
\end{myequation}
Under certain assumptions, the rate of convergence of the minimizing problem with respect to the smoothing parameter can be quantified as $\mathcal{O}(\varepsilon^\frac{1}{2})$ \cite[Theorem 4.13]{Kouri2016cvar}.

\subsection{Proper orthogonal decomposition reduced-order modeling}\label{pod}

Solving the optimization problem (\ref{Jbetaalpha-eps}) requires a vast number of Helmholtz solutions $p(\mathbf{x},\vartheta;\xi)$ computed for different values of $\vartheta$ and $\xi$. In this section, we describe the use of the proper orthogonal decomposition (POD) approach to dramatically reduce the computational cost. A deterministic solution of (\ref{PDEstate}) can be determined by solving the problem:
find $\widetilde{p} = p -\mu p_g \in V_0=H^1_{\Gamma_1}(D;\mathbb{C})$, such that
\begin{equation}\label{liftDetermi}
a_\vartheta(\widetilde{p},\phi_D) = b_\vartheta(\phi_D) \quad \forall \phi_D \in V_0,
\end{equation}
where $p_g$ is an auxiliary function introduced to render the Dirichlet condition homogeneous and $b_\vartheta(\cdot)$ is a bounded linear functional; see Section~\ref{wellposeAna}.
To simplify the discussion, in this section we focus on the alternative solution $\widetilde{p}(\mathbf{x},\vartheta;\xi)$. 

Given a pre-specified parameter value $\vartheta$ and control value $\xi$, a spatial finite element solution to (\ref{liftDetermi}) is represented as $\widetilde{p}_h(\mathbf{x},\vartheta;\xi) = \sum\limits_{j=1}^n \widetilde{p}_j(\vartheta,\xi)\phi_j(\mathbf{x})$. 
Obtaining a high-fidelity solution $\widetilde{p}_h$ usually requires a large number of degrees of freedom $n$ because the finite element basis $\{ \phi_j(\mathbf{x})\}$ is of general purpose and does not contain any information on the problem at hand. The derived discrete system is of large dimension of order $n$, has complex entries, and is indefinite and thus is computationally demanding. 
The goal is to drastically reduce the dimension of the algebraic system (so that it can be solved cheaply) whereas not losing too much accuracy.

\textbf{Reduced-order modeling.}
The idea is to construct a small set of basis functions $\{{\varphi}_i \}_{i=1}^N$ in the finite element space $V_h$
such that the solution $\widetilde{p}(\mathbf{x},\vartheta;\xi)$ can be well approximated in the space $V_R = \spn\{{\varphi}_i\}$, referred to as the {\it reduced space}. The functions $\{{\varphi}_i \}_{i=1}^N$ form the {\it reduced basis} (RB). To construct a {\it reduced-order model} (ROM) for the Helmholtz equation, we impose Galerkin  projection onto the reduced space: find $\widetilde{p}_R\in V_R$ satisfying
\begin{equation}\label{ROMproj}
a_{\vartheta}(\widetilde{p}_R,\varphi_R) = b_\vartheta(\varphi_R)\quad \forall \varphi_R\in V_R.
\end{equation}
Let $\widetilde{p}_R = [{\varphi}_1, \cdots, {\varphi}_N] \widetilde{\mathbf{p}}_{R}(\vartheta, \xi)$ with $\widetilde{\mathbf{p}}_{R}(\vartheta, \xi)$ denoting the vector of coordinates in the reduced space. Substituting this representation into (\ref{ROMproj}) we obtain the reduced system
\begin{equation}\label{generalReduced}
\Big[a_\vartheta(\varphi_j, \varphi_i)\Big]_{N\times N} \widetilde{\mathbf{p}}_{R} 
= \Big[b_\vartheta(\varphi_i)\Big]_{N\times 1}.
\end{equation}
This linear system features a very small size $N$ so that it can be efficiently tackled using a direct solver. 

The RB $\{{\varphi}_i \}_{i=1}^N$ can be constructed from the full finite element approximation of (\ref{liftDetermi}), which we refer to as the {\it full-order model} (FOM). The computation is of large scale and therefore expensive, so it is performed offline. 
In the online phase, the ROM (\ref{generalReduced}) is solved many times for different values of $\vartheta$ and $\xi$, incurring remarkably lower computational costs compared to that for FOM.

\textbf{POD basis construction.} 
RB construction starts by sampling the varying parameter vector $\vec\nu=[\vartheta,\xi]$. Let the sample set $\Xi_{\rm smp} = \{\vec\nu_1, \ldots, \vec\nu_m \}$
consist of $m$ distinct sample of $\vec\nu$. 
A RB is constructed so to guarantee that for each $\vec\nu_i \in \Xi_{\rm smp}$, the error of approximating $\widetilde{p}(\mathbf{x},\vec\nu_i)$ in the reduced space is bounded by a desired tolerance. We follow the proper orthogonal decomposition (POD) approach, which constructs, in a certain sense, an ``optimal" reduced basis as specified below.

For the sake of completeness, we briefly some recall basic features of POD; full details can be found in \cite{Kunisch2001}.
Given the parameter sample set $\Xi_{\rm smp}$, we solve the FOM for each parameter value in $\Xi_{\rm smp}$. These solutions are referred to as {\it snapshots} and are denoted by $\{\widetilde{p}^i_{S,h} \}_{i=1}^m$. We treat $\widetilde{p}^i_{S,h}$ as a two-dimensional vector of real functions. The POD approach seeks an orthonormal {\it POD basis} $\{{\varphi}_1, \cdots, {\varphi}_N \}$ (also known as the set of {\it POD modes}) in $V_h$ of a given rank $N~(N\ll m)$ that can best approximate the training space $V^{trn} = \spn\{\widetilde{p}^i_{S,h} \}_{i=1}^m$. Here, ``best'' means that the POD basis solves 
\begin{equation}\label{minPOD}
\min_{\{{\psi}_i\}} \sum_{j=1}^m \big\Vert\widetilde{p}^j_{S,h}-\sum_{i=1}^N\langle \widetilde{p}^j_{S,h},{\psi}_i \rangle_{L^2}{\psi}_i \big\Vert^2_{L^2}
 \qquad\text{ s.t. } \langle{\psi}_i,{\psi}_j\rangle_{L^2} = \delta_{ij}.
\end{equation}
Denote by $\widetilde{\mathbf{p}}^j_S \in \mathbb{R}^{2n}=\mathbb{C}^n$ the vector of finite element coefficients of $\widetilde{p}^j_{S,h}$. We gather the snapshot vectors into as columns of the {\it snapshot matrix} $\mathbf{P}=[\widetilde{\mathbf{p}}^1_S, \cdots, \widetilde{\mathbf{p}}^m_S] \in \mathbb{R}^{2n\times m}$. We also introduce the correlation matrix $\mathbf{C}_{\rm P}=\big[<\widetilde{p}^j_{S,h}, \widetilde{p}^i_{S,h}>_{L^2} \big] = \mathbf{P}^T\mathbf{M}P \in \mathbb{R}^{m\times m}$ corresponding to the snapshots, where $\mathbf{M}$ denotes the mass matrix. The POD basis is then constructed by determining the eigenvectors of $\mathbf{C}_{\rm P}$:
\begin{itemize}
\item[] let $(\lambda_j,\mathbf{u}_j)$ denote the eigen-pairs of $\mathbf{C}_{\rm P}$ with $\lambda_1 \geq \cdots \geq \lambda_N\geq\cdots\geq\lambda_d>0$ ($d=\mbox{rank }\mathbf{C}_{\rm P}$) and $\mathbf{u}_i^T\mathbf{u}_j = \delta_{ij}$; then, the POD mode $\varphi_i$ has finite element representation $\frac{1}{\sqrt{\lambda_i}}\mathbf{P}\mathbf{u}_i$ \cite{Kunisch2001}.
\end{itemize}
From \cite{Kunisch2001}, we also have the $L^2$ error estimate
\begin{equation}\label{L2PODerr}
\sum_{j=1}^m \big\Vert\widetilde{p}^j_{S,h}-\sum_{i=1}^N\langle \widetilde{p}^j_{S,h},{\varphi}_i \rangle_{L^2}{\varphi}_i \big\Vert^2_{L^2} =  \sum_{i=N+1}^{d}\lambda_i.
\end{equation}
The $H^1$ norm error estimate can be derived from \cite[Lemma 3.2]{Iliescu2014}:
\begin{equation}\label{H1PODprojErr}
\sum_{j=1}^m \big\Vert\widetilde{p}^j_{S,h}-\sum_{i=1}^N\langle \widetilde{p}^j_{S,h},{\varphi}_i \rangle_{L^2}{\varphi}_i \big\Vert^2_{H^1} =  \sum_{i=N+1}^{d}\lambda_i\norm{\varphi_i}_{H^1}^2.
\end{equation}

In practice, one may construct the POD basis on the finite element vectors in  Euclidean space, which could be more efficient. In this case, the POD modes are given by the $N$ left singular vectors of $\mathbf{P}$ associated with the $N$ largest singular values. An efficient way for computing them 
is to first compute the {\it thin QR factorization} of $\mathbf{P}$ as $\mathbf{P} = \mathbf{QR}$ , and then compute the singular value decomposition of the small matrix $\mathbf{R}\in\mathbb{R}^{m\times m}$ as $\mathbf{R}=\mathbf{U}_{\rm R}\mathbf{S}_{\rm R}\mathbf{V}_{\rm R}^T$. The POD modes can be extracted in order from the columns of $\mathbf{QU}_{\rm R}$.


\section{A parallel POD-BFGS optimization method}\label{numScheme}

Monte Carlo methods \cite{Fishman1996} are among the most popular choices for approximating statistical moments such as expectations of solutions of SPDEs.  Applied to the stochastic Helmholtz equation, a classical Monte Carlo method (MCM) proceeds by 
\begin{itemize}
\item[--] drawing  $Q$ independent and identically distributed (i.i.d.) random samples of $\vartheta$: $\{\vartheta_1,\ldots,\vartheta_Q \}$;
\item[--] computing a spatial finite element solution $p_h(\mathbf{x}, \vartheta_j )$ for each sample point $\vartheta_j$;
\item[--] obtaining the statistical quantity of interest by averaging over the $Q$ realizations, e.g., \[\int_\Lambda h_\varepsilon\Big(\frac{1}{\gamma_{\rm p}}\int_D|p_h(\mathbf{x},\vartheta)|^2d\mathbf{x} - \alpha \Big)\rho(\vartheta)d\vartheta 
\approx \frac{1}{Q}\sum\limits_{j=1}^Q h_\varepsilon\Big(\frac{1}{\gamma_{\rm p}}\int_D|p_h(\mathbf{x},\vartheta_j)|^2d\mathbf{x} - \alpha \Big).\]
\end{itemize} 
As a non-intrusive approach, MCM requires only deterministic Helmholtz solutions for each realization. Although the numerical error of MCM is proportional to $1/\sqrt{Q}$, thus requiring a large number of realizations, the convergence behavior holds true for any dimension of the random vector in the SPDE. In this sense, it avoids the curse of dimensionality.
Moreover, because here we solve the Helmholtz equation in the reduced-order space having small dimension, obtaining a large number of Helmholtz solutions is not nearly as challenging an endeavor.

In the rest of this section, we describe the parallel processing schemes we use for the state and optimization problems.

\subsection{The reduced Helmholtz solver in martix form}\label{sec:HelmROM}
For specified values of $\vartheta$ (i.e.~$[k,\mu_{\rm r},\mu_{\rm i}]$) and $\xi$, we solve the deterministic full-order Helmholtz equation by the finite element (FE) method. 
Letting $\{\phi_j\}_{j=1}^n$ be a finite element basis on the acoustic domain $D$ associated with the nodes $\{\mathbf{x}_j\}_{j=1}^n$, we define the mass matrix $\mathbf{M}^0$, the stiffness matrix $\mathbf{S}^0$, and the boundary mass matrices $\mathbf{K}^0_2$ and $\mathbf{K}^0_4$ with entries as follows:
\begin{align*}
& [\mathbf{M}^0]_{jk} = \int_D \phi_k\phi_j d\mathbf{x} \qquad [\mathbf{S}^0]_{jk} = \int_D \nabla\phi_k\cdot\nabla\phi_j d\mathbf{x} \\
& [\mathbf{K}^0_2]_{jk} = \int_{\Gamma_2} \phi_k\phi_j ds \qquad [\mathbf{K}^0_4]_{jk} = \int_{\Gamma_4} \phi_k\phi_j ds. &
\end{align*}
We use the bold symbol $\mathbf{p}\in \mathbb{C}^n$ to denote the vector representation of of the coefficients of $p(\mathbf{x})$ in the finite element space, and analogously for other functions. If we ignore the Dirichlet boundary condition for the time being, the FE discretization of the Helmholtz model is formulated, in matrix form, as
\begin{equation*}
\Big(\mathbf{S}^0 - k^2\mathbf{M}^0 + \frac{ik(\xi_{\rm r}-i\xi_{\rm i})}{|\xi|^2}\mathbf{K}^0_2 + ik\mathbf{K}^0_4\Big)\mathbf{p} = \boldsymbol{0}.
\end{equation*}
In the simulation, we split the real and imaginary parts of complex-valued functions $p = p_{\rm r} + i p_{\rm i}$ and $\xi = \xi_{\rm r} + i \xi_{\rm i}$, but still use $\mathbf{p}\in \mathbb{R}^{2n}$ to denote the splitting form $[\mathbf{p}_{\rm r}, \mathbf{p}_{\rm i}]^T$ by an abuse of notation. Then, the algebraic system in terms of real-valued variables is given by
$\widetilde{\mathbf{A}}\mathbf{p} = \boldsymbol{0}$,
where
\begin{equation*}
\widetilde{\mathbf{A}} = 
\begin{bmatrix}
\mathbf{S}^0 - k^2\mathbf{M}^0 + \frac{k\xi_{\rm i}}{|\xi|^2}\mathbf{K}^0_2 & -\frac{k\xi_{\rm r}}{|\xi|^2}\mathbf{K}^0_2 - k\mathbf{K}^0_4 \\[0.3cm]
\frac{k\xi_{\rm r}}{|\xi|^2}\mathbf{K}^0_2 + k\mathbf{K}^0_4 & \mathbf{S}^0 - k^2\mathbf{M}^0 + \frac{k\xi_{\rm i}}{|\xi|^2}\mathbf{K}^0_2
\end{bmatrix}.
\end{equation*}
Imposing the Dirichlet boundary condition on $\Gamma_1$ leads to the system
\begin{equation}\label{disHelm}
\mathbf{A}\mathbf{p}=\mathbf{b}
\end{equation}
with $\mathbf{A} = (\mathbf{I}-\mathbf{I}_{\Gamma_1})\widetilde{\mathbf{A}} + \mathbf{I}_{\Gamma_1}$ and $\mathbf{b} = \mu_{\rm r}\mathbf{g}_{\rm r} + \mu_{\rm i}\mathbf{g}_{\rm i}$. Here, the matrix $\mathbf{I}_{\Gamma_1}$ marks the indices on the boundary $\Gamma_1$: 
the $i$-th row of $\mathbf{I}_{\Gamma_1}$ is the unit vector $\mathbf{e}_i^T$ in $\mathbb{R}^{2n}$ if $\mathbf{x}_i \text{ or }\mathbf{x}_{i-n}\in \Gamma_1$ and zero else. We assume that $g_{\Gamma 1}(\mathbf{x})$ is a real function and denote its finite element interpolation in $\mathbb{R}^n$ as $\mathbf{g}_{\Gamma_1}$. 
The vector $\mathbf{g}_{\rm r}$ in $\mathbf{b}$ is a two-component vector with the first component being $\mathbf{g}_{\Gamma_1}$ and the second being zero. The vector $\mathbf{g}_{\rm i}$ is obtained instead by putting zero to the first component and $\mathbf{g}_{\Gamma_1}$ to the second.

We implemented the Helmholtz solver in LifeV, which is an object oriented parallel finite element library in C++ developed by several groups worldwide ({\tt{www.lifev.org}}). For convenience, we treat the complex-valued function $p(\mathbf{x})$ as a two dimensional vector function and build its corresponding finite element space. In such case, the mass, stiffness, and boundary mass matrices are in the form
$\mathbf{W} = 
\begin{bmatrix}
\mathbf{W}^0  & \\
 & \mathbf{W}^0
\end{bmatrix}
$
where $\mathbf{W}$ stands for $\mathbf{M}, ~\mathbf{S}, ~\mathbf{K}_2$, and $\mathbf{K}_4$.
We further define the ``skew'' boundary mass matrix 
$\widetilde{\mathbf{K}}_2 = 
\begin{bmatrix}
 & -\mathbf{K}_2^0  \\
 \mathbf{K}_2^0 & 
\end{bmatrix}
$ and $\widetilde{\mathbf{K}}_4$ analogously.
The block matrix $\widetilde{\mathbf{A}}$ can then be expressed as
\begin{equation*}
\widetilde{\mathbf{A}} = \mathbf{S}-k^2\mathbf{M}+\frac{k\xi_{\rm i}}{|\xi|^2}\mathbf{K}_2 + \frac{k\xi_{\rm r}}{|\xi|^2}\widetilde{\mathbf{K}}_2 + k\widetilde{\mathbf{K}}_4.
\end{equation*}

The linear algebraic system (\ref{disHelm}) for the Helmholtz equation is indefinite so that many classical iterative methods encounter convergence issues. However, some iterative method  \cite{Erlangga2008} for this setting have been developed. In particular, the shifted Laplacian preconditioner for Krylov subspace methods has attracted much attention for its relative robustness and efficiency \cite{Laird_preconditionediterative, Erlangga2004409, Airaksinen20101796, Gander2015}. Inspired by these developments, we construct a preconditioner from a discrete version of the shifted Laplacian operator $-\Delta  - (\beta_1-\beta_2 i)k^2 $ (subject to the same boundary conditions as in (\ref{PDEstate})) for some $\beta_1, \beta_2 \in \mathbb{R}$. According to our numerical experience, the value $(\beta_1, \beta_2) = (1, 0.5)$ performs well, an incomplete LU approximation of the discrete shifted Laplacian operator will be taken as the preconditioner in our simulations. The preconditioned Helmholtz system is finally solved by the GMRES iterative method implemented in the Trilinos package ({\tt{www.trilinos.org}}). 
To maintain the scalability of the ILU preconditioner in a parallel computing enviroment, we take a local ILU factorization with overlap level 4 (see the IFPACK package in Trilinos).

We next build the reduced-order Helmholtz solver. 
Letting $\mathbb{Z} \in \mathbb{R}^{2n\times N}$ be the reduced basis for the acoustic pressure $p$, we represent $\mathbf{p}$ in the full-order space as $\mathbb{Z}\mathbf{p}_{\rm rb}$.
By projecting the discrete Helmholtz system (\ref{disHelm}) onto the reduced space spanned by $\mathbb{Z}$, we obtain the reduced Helmholtz model
\begin{equation}\label{ROMhelm}
\underbrace{\mathbb{Z}^T\mathbf{A}\mathbb{Z}}_{\mathbf{A_r}}\mathbf{p}_{\rm rb} = \underbrace{\mathbb{Z}^T\mathbf{b}}_{\mathbf{b}_{\rm r}}
\end{equation}
for the reduced-basis solution $\mathbf{p}_{\rm rb}$.
Note that $\mathbf{A_r} \in \mathbb{R}^{N\times N}$ is a dense matrix but in general it has a very small size, hence the linear system (\ref{ROMhelm}) can be tackled with a direct solver. 

In practice, if we store in the offline computation the small dense matrices $\mathbf{M}_{\rm r}=\mathbb{Z}^T(\mathbf{I}-\mathbf{I}_{\Gamma_1})\mathbf{M}\mathbb{Z}$ (and $\mathbf{S}_{\rm r}, \mathbf{K}_{\rm 2r}, \widetilde{\mathbf{K}}_{\rm 2r}, \widetilde{\mathbf{K}}_{\rm 4r}$ analogously) and $\mathbf{I}_{\rm r}=\mathbb{Z}^T\mathbf{I}_{\Gamma_1}\mathbb{Z}$, we can efficiently (in work depending on $N$ and not on the dimension of the finite element space) assemble the reduced coefficient matrix online as
$$
\mathbf{A}_{\rm r} = \mathbf{S}_{\rm r}-k^2\mathbf{M}_{\rm r}+\frac{k\xi_{\rm i}}{|\xi|^2}\mathbf{K}_{\rm 2r} + \frac{k\xi_{\rm r}}{|\xi|^2}\widetilde{\mathbf{K}}_{\rm 2r} + k\widetilde{\mathbf{K}}_{\rm 4r} + \mathbf{I}_{\rm r}.
$$
Similarly, the reduced right hand side in (\ref{ROMhelm}) can be assembled online without dependence on the full-order size.
%

\subsection{The algorithm for parallel POD-BFGS optimization}\label{sec:podOpt}

In computing the CVaR measure, one can approximate the integral $\int_\Lambda \cdot ~\rho(\vartheta)d\vartheta$ numerically by any appropriate quadrature rules, say with quadrature weights $\omega_1, \cdots, \omega_Q \in \mathbb{R}$ and quadrature points $\vartheta_1,\cdots, \vartheta_Q \in \Lambda$. Due to the efficiency of the reduced-order model, the Monte Carlo method ($\omega_j = \frac{1}{Q}$) is a natural choice for handling this integral despite the large number $Q$ of realizations (deterministic problems to be solved) required for high accuracy.
In the cost function, $\int_D |p(\mathbf{x})|^2d\mathbf{x} = \mathbf{p}^T\mathbf{M}\mathbf{p}$ characterizes the amount of noise.
The full-order discrete optimization problem is given by
\begin{equation}\label{J-discr}
\min\limits_{\xi,\alpha}J^Q(\xi,\alpha) =\frac{1}{2}\bigg[ \alpha + \frac{1}{1-\beta}\sum\limits_{j=1}^Q\omega_jh_\varepsilon\big(\frac{1}{\gamma_{\rm p}} \int_D |p(\mathbf{x},\vartheta_j;\xi)|^2d\mathbf{x}-\alpha \big) \bigg] +\frac{\gamma}{2}|\xi|^2 
\end{equation}
In reduced order, the noise energy $\int_D |p(\mathbf{x})|^2d\mathbf{x}$ equals $\Vert \mathbf{p}_{\rm rb}\Vert_{2}$ due to the orthogonality of the reduced basis. The corresponding reduced-order optimization problem is given by
\begin{myequation}\label{Jr-discr}
\min\limits_{\xi,\alpha}J_{\rm r}^Q(\xi,\alpha) =\frac{1}{2}\bigg[ \alpha + \frac{1}{1-\beta}\sum\limits_{j=1}^Q\omega_jh_\varepsilon\big(\Vert\mathbf{p}_{\rm rb}(\vartheta_j;\xi)\Vert_{2}^2/\gamma_{\rm p}-\alpha \big)\bigg] +\frac{\gamma}{2}|\xi|^2, 
\end{myequation}
where $\mathbf{p}_{\rm rb}$ solves (\ref{ROMhelm}). This is the problem we will solve in Section~\ref{sec:resultCVaR}.

The cost function $J_{\rm r}^Q$ can be evaluated efficiently using parallel computing. Because the ROM is of small size, 
the system assembly from the pre-stored matrices $\mathbf{M}_{\rm r}$, $\mathbf{S}_{\rm r}$, etc.~can be executed independently on each processor. Communication among processors only occurs when the noise energies need to be summed to form $J_{\rm r}^Q$. A graphic description of the parallel computation of $J_{\rm r}^Q$ is illustrated in Figure~\ref{mpi-pic}.

\begin{figure}[h!]
\begin{center}
\includegraphics[scale=0.4]{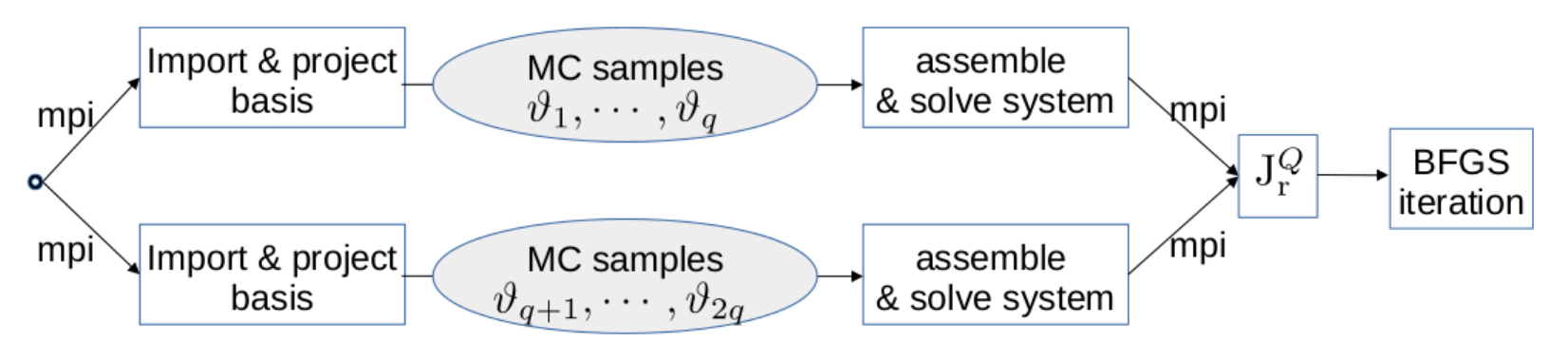}
\end{center}
\caption{Parallel computation of $J_{\rm r}^Q$}
\label{mpi-pic}
\end{figure}

For the application of a gradient-based optimization solver, we introduce the Lagrangian multipliers $\{\mathbf{q}_j\}_{j=1}^Q \subseteq \mathbb{R}^N$ 
that solve
the adjoint systems
\begin{equation*}
\mathbf{A}_{\rm r}(\vartheta_j)^T\mathbf{q}_j = \frac{h_\varepsilon^\prime\big(\Vert\mathbf{p}_{\rm rb}(\vartheta_j;\xi)\Vert_{2}^2/\gamma_{\rm p}-\alpha \big)}{(1-\beta)\gamma_{\rm p}}\mathbf{p}_{\rm rb}(\vartheta_j),\quad j=1,\cdots,Q.
\end{equation*}
The  G\^ateaux derivative of $J_{\rm r}^Q(\xi,\alpha)$ can then be formulated as
\begin{align}
&\frac{D J_{\rm r}^Q(\xi)}{D\xi} = \gamma\xi-\sum\limits_{j=1}^Q\omega_j\mathbf{q}_j^T\frac{\partial \mathbf{A}_{\rm r}}{\partial \xi}
(\vartheta_j)\mathbf{p}_{\rm rb}(\vartheta_j) \\
&\frac{D J_{\rm r}^Q(\xi)}{D\alpha} = \frac{1}{2} - \frac{1}{2(1-\beta)}\sum\limits_{j=1}^Q\omega_jh^\prime_\varepsilon\big(\Vert\mathbf{p}_{\rm rb}(\vartheta_j; \xi)\Vert_{2}^2/\gamma_{\rm p}-\alpha\big).
\end{align}

We solve the optimization problem (\ref{Jr-discr}) by the BFGS quasi-Newton method as shown in Algorithm \ref{opt-alg}. The line search in step \ref{linesch} of Algorithm \ref{opt-alg} is based on  cubic interpolation of the misfit function and on the Armijo condition (\cite{nocedal2006}). In practice, the line search performs well enough with a at most two iterations. In the algorithm, we assume that the total number $Q$ of Monte Carlo samples is a multiple of the number of processors. 

\begin{algorithm}[tp]
\caption{Parallel POD-BFGS Optimization}\label{opt-alg}
\begin{algorithmic}[1]
\Input initial guess $\mathbf{\xi}^0$ and $\alpha^0$, probability level $\beta$, POD basis $\mathbb{Z}$, number of processors $n_{\rm p}$
\Output estimated impedance value and its corresponding $\mbox{VaR}_\beta$
\State Import and project the basis $\mathbb{Z}$ (communication among processors)
\State On processor $i\in\{0,\cdots,n_p-1\}$, draw i.i.d.~random samples $\Theta^i = \{\vartheta_{1+i\frac{Q}{n_{\rm p}}},\cdots,\vartheta_{(i+1)\frac{Q}{n_{\rm p}}} \}$ of $\vartheta$
\State Initialize inverse Hessian $\boldsymbol{H}_0 \leftarrow \frac{1}{||\nabla J_{\rm r}^Q(\xi^0,\alpha^0)||}\mathbf{I} $
\State $k\leftarrow 0$
\While{stopping criterion not satisfied}
		\State On processor $i\in\{0,\cdots,n_p-1\}$, solve $\mathbf{p}(\vartheta;\xi^k)$ for $\vartheta\in\Theta^i$
		\State Evaluate $J_{\rm r}^Q(\xi^k,\alpha^k)$ (communication among processors)
		\State On processor $i\in\{0,\cdots,n_p-1\}$, solve $\mathbf{q}_j$ for $j\in \{1+i\frac{Q}{n_{\rm p}}, \cdots, (i+1)\frac{Q}{n_{\rm p}} \}$
		\State Evaluate $\nabla J_{\rm r}^Q(\xi^k,\alpha^k)$ (communication among processors)		
		\State Compute search direction $\mathbf{v}^k = -\boldsymbol{H}_k\nabla J_{\rm r}^Q(\xi^k,\alpha^k) $ 
				
\vspace{0.1cm}		
		  
		\State Set $[{\xi}^{k+1},\alpha^{k+1}]^T = [{\xi}^{k},\alpha^{k}]^T + \gamma_k \mathbf{v}^k$ with $\gamma_k \in (0, \infty)$ computed from a line search \label{linesch}
						
\vspace{0.1cm}		
		  
		\State Define $\mathbf{s}_k = [{\xi}^{k+1},\alpha^{k+1}]^T-[{\xi}^{k},\alpha^{k}]^T$, $\mathbf{y}_k=\nabla J_{\rm r}^Q({\xi}^{k+1},\alpha^{k+1}) - \nabla J_{\rm r}^Q({\xi}^{k},\alpha^{k})$, $\rho_k = \frac{1}{\mathbf{y}_k^T\mathbf{s}_k}$
						
\vspace{0.1cm}		
		  
		\State Update the inverse Hessian
		
\vspace{-0.3cm}		
		  
$$\boldsymbol{H}_{k+1} = (\mathbf{I} - \rho_k \mathbf{s}_k \mathbf{y}_k ^T)\mathbf{H}_k(\mathbf{I} - \rho_k \mathbf{y}_k \mathbf{s}_k^T ) + \rho_k \mathbf{s}_k \mathbf{s}_k^T \qquad\mbox{(BFGS \cite{nocedal2006})}
 $$
 		
\vspace{-0.2cm}		
		  
		\State $k\leftarrow k+1$
\EndWhile
\State \textbf{return} ${\xi}^{k}$ and $\alpha^k$
\end{algorithmic}
\end{algorithm}

\section{Numerical experiments for impedance optimization for turbofan noise reduction}\label{sec:resutls}

In this section, we apply the POD-BFGS optimization strategy for impedance optimization for turbofan noise reduction. We first study the efficiency and accuracy of the POD-based Helmholtz solver after which we provide computed optimal impedance values obtained by minimizing, in different scenarios, the CVaR measure. 

In this study, simulations of fan noise propagation are performed on a realistic intake geometry.
Figure~\ref{2dGeo} (left) shows a two-dimensional section of a High Bypass Ratio (HBR) turbofan engine and its nacelle intake geometry (middle) used to create the axi-symmetric acoustic computational domain. To the right, it shows the three-dimensional geometry (zoomed in) with a typical example of noise radiation from the fan plane. 
The domain and its mesh (shown in Figure~\ref{mesh3D}) are generated using the Gmsh software; it matches the key geometric data created by the intake model used for various liner studies \cite[Figure~5.4]{thesisMustafi} in the European Commission research project SILENCE(R)\footnote{http://www.xnoise.eu/index.php?id=85}.
Specifically, the fan radius is 1.2m. The acoustic liner, highlighted in red in Figure~\ref{2dGeo}, has a length of 1.08m starting at a distance of 0.21m from the fan plane. The far field boundary $\Gamma_4$ is 5m away from the fan noise source. 

\begin{figure}[h!]
\begin{center}
\begin{minipage}{0.28\textwidth}
\includegraphics[scale=0.35]{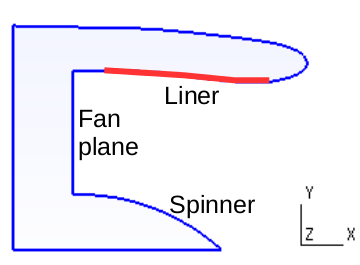}
\end{minipage}
\begin{minipage}{0.38\textwidth}
\includegraphics[scale=0.4]{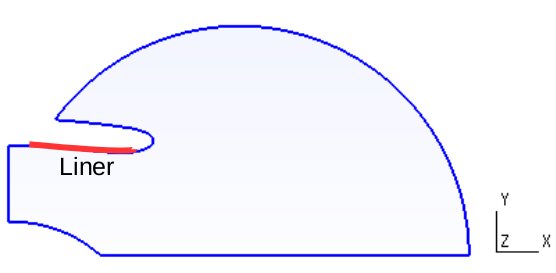}
\end{minipage}
\quad
\begin{minipage}{0.25\textwidth}
\includegraphics[scale=0.2]{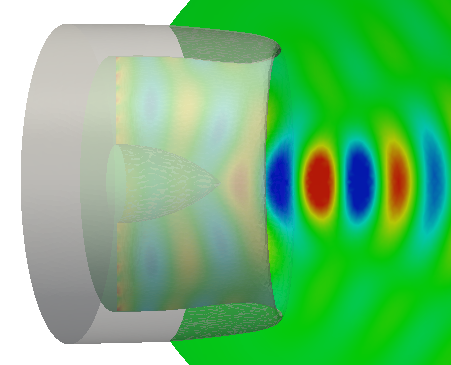}
\end{minipage}
\caption{Two-dimensional sections of turbofan engine geometry (left), nacelle intake (middle), and acoustic liner (indicated in red). The right picture shows the three-dimensional geometry with a typical example of noise radiation from the fan plane.	}
\label{2dGeo}
\end{center}
\end{figure}

Because most of the aircraft noise energy is in the low-frequency range\cite{Leventhall2003},
we typically set the wavenumber range as $\Lambda_1 = [5, 10] \subseteq \mathbb{R}$.
To have adequate mesh resolution to ensure accuracy for wavelength $\lambda=2\pi/k \in [0.2\pi, 0.4\pi]$, generally 10 points per wavelength is sufficient \cite{MR1639879, Thompson1994}. In the meshes we use, we set the characteristic length to be 0.05 and obtain a tetrahedral mesh of the acoustic domain with 152,891 vertices and 819,554 elements.

We assume that the wavenumber and the sound noise amplitude are uniformly distributed, and in particular $\mu_{\rm r}$ and $\mu_{\rm i}$ range from 10 to 30. Furthermore, we take the noise source $g_{\Gamma_1}(\mathbf{x})= 1+\sqrt{y^2+z^2}\cos\big(10\pi(y+z)\big)$, depicted in Figure~\ref{g_profile_andSVD} (left). This function is chosen so that the sound pressure level is higher around the fan than near the axis.

\begin{figure}[h!]
\centerline{
\includegraphics[scale=0.35]{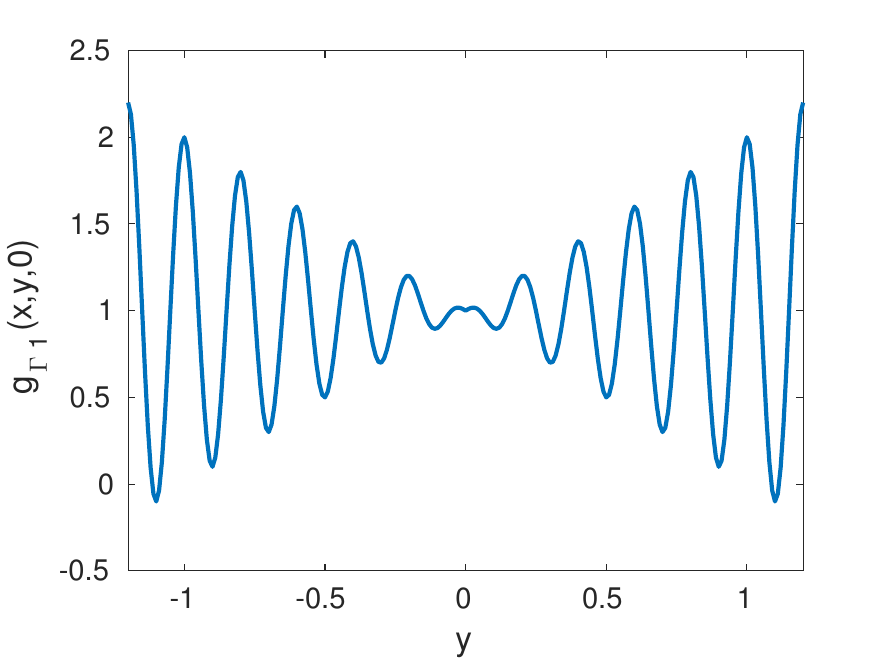}
\qquad
\includegraphics[scale=0.35]{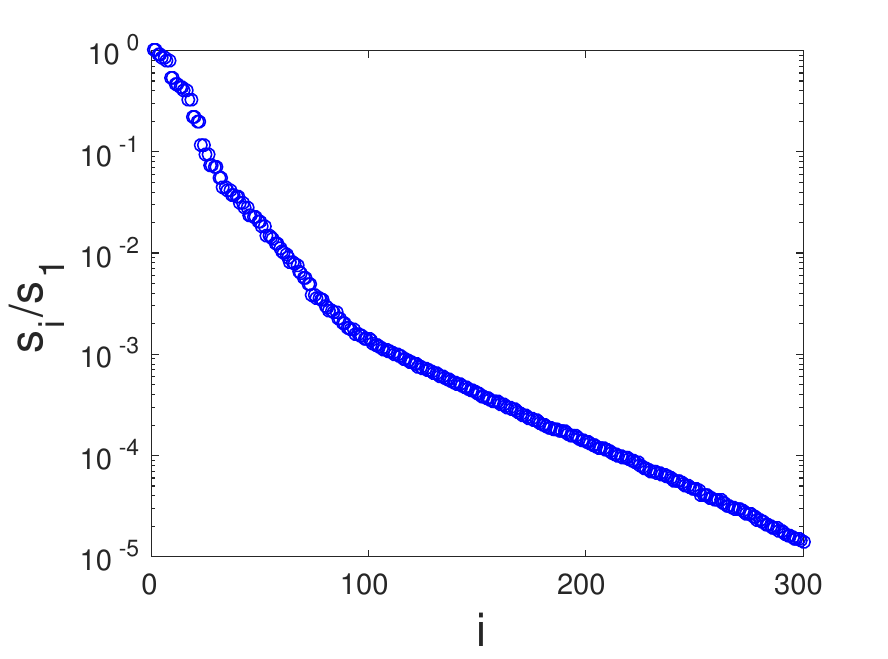}
}
\caption{Left: profile of the noise source term $g_{\Gamma_1}(\mathbf{x})$ with the the z-coordinate fixed to zero. Right: the leading 300 singular values ($s_i = \sqrt{\lambda_i}$) of the snapshot matrix for the acoustic pressure $p$.}
\label{g_profile_andSVD}
\end{figure}

We now study the efficiency and the parallel scaling performance of the full-order Helmholtz solver in an un-thorough manner because our main focus is not on the full-order state problem. The deterministic finite-element solutions of the Helmholtz equation (\ref{disHelm}), for a fixed input of $(\xi_{\rm r},\xi_{\rm i},k,\mu_{\rm r},\mu_{\rm i})$, are computed with different number of processors given a relative tolerance $10^{-6}$. Each processor is an Intel(R) Xeon(R) CPU E5-2670 @2.60GHz.
The average execution time of several distinct realizations are displayed in Table \ref{fomParallelTimeTable}. It splits into two sequential steps: the time for building the finite element matrices ($\mathbf{M}, \mathbf{S}, \mathbf{K}_2, \widetilde{\mathbf{K}}_2, \widetilde{\mathbf{K}}_4, \mathbf{I}_{\Gamma_1}$) and the time for final assembly and solving the system $\mathbf{Ap}=\mathbf{b}$. The strong scaling efficiency (Eff.) given in the table, as a percentage of linear,  is the ratio of the amount of time with one processor to the product of the amount of time with multiple processors and the number of processors. 

In building the FE matrices, the solver features super-linear parallel scaling efficiency (\mbox{Eff. }$> 1$) mostly because of the super-linear speedup of RAM access time, whereas this step is executed only once for a stochastic solution. 
Although the system assembly and solving step has good scalability with few processors, the computation is still expensive if it needs to be executed tens of thousands of times in a stochastic sampling or collocation method for SPDEs. The computation is even more intensive in an optimization problem, where the stochastic Helmholtz equation and its adjoint counterpart would be iterated hundreds of times. This motivates us to apply model-order reduction techniques (still in parallel) with the aim of significantly reducing the computational cost.

\begin{table}[h!]
\begin{center}
\caption{Execution time of the full-order Helmholtz solver with different number of processors. Eff.~measures the strong scaling efficiency as a percentage of linear.}
\begin{tabular}{clll}
\toprule
\# Processors	& \multicolumn{2}{c}{CPU time (sec.)} \\
    \cmidrule(l){2-3}
				 & build FE matrices (Eff.)&  assemble \& solve system (Eff.) \\
\midrule
16 & 2.994	(404.6\%) &	22.49 (55.3\%)\\
8 & 6.180	(392.1\%)	&	31.73 (78.5\%) \\
4 & 15.25 (317.8\%)	&	51.42 (96.8\%)\\
2 & 64.84	(149.5\%) &	99.87 (99.7\%)\\
1 & 193.84	&	199.14 \\
\bottomrule
\end{tabular}
\label{fomParallelTimeTable}
\end{center}
\end{table}

\subsection{Performance of the reduced-order model} 

The reliability of the reduced stochastic optimization problem (\ref{Jr-discr}) depends on the accuracy of the reduced-order model, and essentially on the reduced basis construction. 
Because we include uncertainties in the wavenumber $k$ and noise source amplitude $\mu$, an effective POD basis for the stochastic Helmholtz model should take into account knowledge associated with these uncertainties. 
Moreover, the basis should contain sensitivities induced by the variation of the impedance (control variable).  
To this end, we take 720 samples for the quadruple $(k,\mu,\xi_{\rm r},\xi_{\rm i})$ from the set
$$ \Xi_{\rm smp} = \{k_1,\ldots, k_{40} \} \times \{ 1, i \} \times \{0.05, 0.5, 2\} \times \{-0.05,-0.5,-2 \},$$
for which the 40 values $\{k_j \}_{j=1}^{40}$ of the wavenumber are uniformly distributed over the interval $[5, 10]$.
The values $\{0.05, 0.5, 2\} \times \{-0.05,-0.5,-2 \}$ of $(\xi_{\rm r},\xi_{\rm i})$ represent the variations of the impedance parameter. 
We take only negative values for $\xi_{\rm i}$  because this is the range of ideal impedance parameter computed from deterministic optimization experiments \cite{Cao2007Estim}.
 Although more samples can be generated by densifyinq the sampling of $k$ and $\xi$ to increase the accuracy of the reduced-order model, our purpose is to keep the number of samples as few as possible so as to guarantee the offline computational efficiency of the POD basis construction.

For each sample in $\Xi_{\rm smp}$, the corresponding deterministic full-order Helmholtz equation is solved offline to construct the snapshot matrix from which the POD basis is determined. Figure~\ref{g_profile_andSVD} (right) plots the 300 largest singular values $s_i$ of the snapshot matrix, scaled by the leading singular value $s_1$. There is no obvious ``kink'' or ``elbow'' in the plot; the slow decay is apparent when compared to other problems which feature an ``L''-shape plot of the singular values, indicating fast decay so that few POD modes are enough for an accurate ROM construction (see, e.g.,~\cite[Figure~3]{BertagnaVeneziani} ). This fact is mainly due to the wave propagation as has been investigated in an electro-physiological problem modeling cardiac potential spreading \cite{HHpodDEIM2016}.

Nevertheless, when a sufficient number of POD modes are included, the corresponding ROM is accurate enough for a vast majority of random input data. This is demonstrated by the box plot of the reconstruction error in Figure~\ref{POD-numMode50pt}, for which the error $e_{\rm rel}$ for the input data $(\xi,\vartheta)$ with POD basis $\mathbb{Z}$ is defined as
$$e_{\rm rel}(\xi,\vartheta;\mathbb{Z}) = \frac{\parallel\mathbb{Z}\mathbf{p}_{\rm rb}(\xi,\vartheta)-\mathbf{p}(\xi,\vartheta)\parallel_{2}}{\parallel\mathbf{p}(\xi,\vartheta)\parallel_{2}}.$$
The seven boxes correspond to the reconstruction error of seven ROMs built with different number of POD modes ranging from 60 to 120, and each box plot is based on 50 realizations with the vector $(k,\mu_{\rm r},\mu_{\rm i},\xi_{\rm r},\xi_{\rm i})$ taking random values from $[5,10]\times[10,30]\times[10,30]\times[0,100]\times[-100,100]$.
Each box spans the first quartile to the third quartile (the interquartile range IQR), and the central red line segment inside shows the median. Points are drawn as outliers (red +) if they are at least 1.5*IQR above the third quartile or  below the first quartile. The ``whiskers'' of each box extend to the most extreme data points which are not outliers.
As we can see, the error median decays below 5\% when the number of POD modes increases to 80 and significantly less when the number is at least 90. 
Although the outliers hold errors around 15\%, they only constitute a minority. Moreover, further experiments show that these large errors generally correspond to local inaccuracy rather than global sound radiation.
We also study the ROM accuracy with ten testing parameter values fixed, the corresponding relative reconstruction error for different ROMs is plotted in the left of Figure~\ref{POD-numMode-fixed}. The fluctuation near 100 is natural because the testing values are ranging largely outside the sampling range for POD basis construction.
Overall, we observe that a ROM constructed with at least 80 POD modes features enough practical accuracy. Henceforth, we  take 90 POD modes for the reduced model in the stochastic optimization problem that follows. Figure~\ref{POD-numMode-fixed} (right) shows the accuracy of the model with 100 random realizations, in which the red dashed line denotes the median. Alternatively, one can choose the number $N$ of POD modes in a way such that
\begin{equation}
\bigg(\sum\limits_{i=1}^N s_i^2\bigg)^{1/2} < \tau \bigg(\sum\limits_{i=1}^d s_i^2\bigg)^{1/2}
\end{equation}
where $\tau \in (0, 1)$ represents the amount of information of the sample $\mathbf{P}=[\widetilde{\mathbf{p}}^1_S, \cdots, \widetilde{\mathbf{p}}^m_S]$ that the POD modes have to capture, by the error estimator (\ref{L2PODerr}). For our problem, $\tau$ is recommended to be at least 0.995 (corresponding to $N=74$) to maintain enough accuracy of the ROM.

\begin{figure}[!h]
\begin{center}
\includegraphics[scale=0.5]{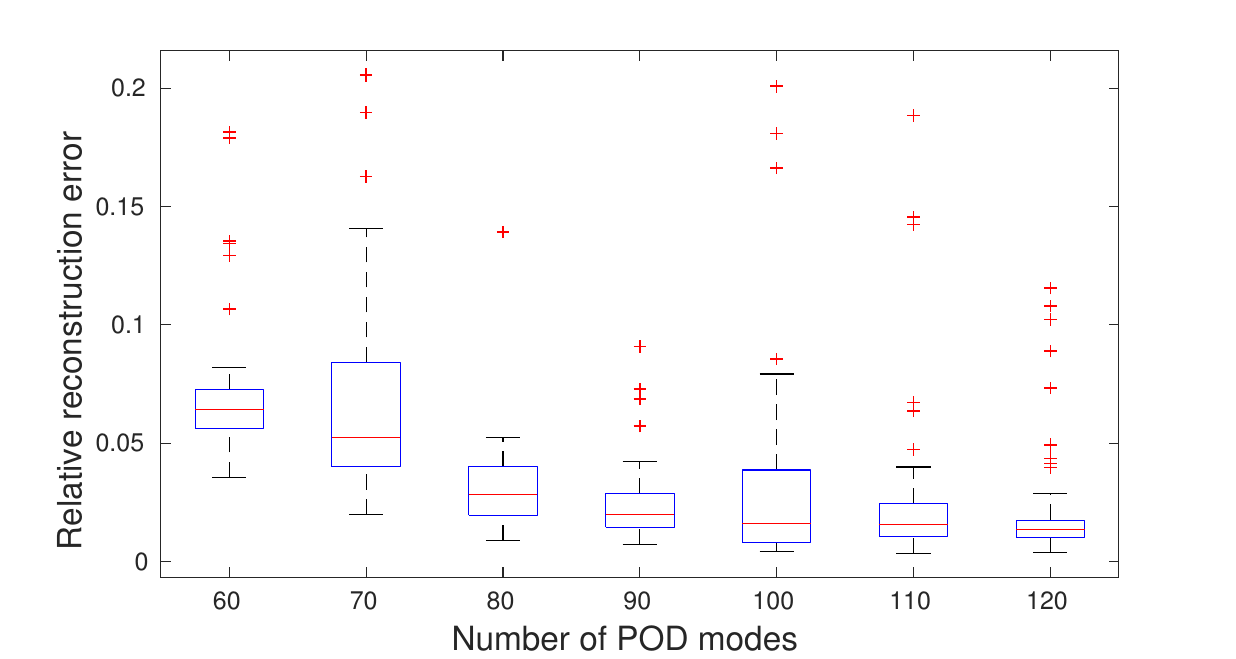} 
\caption{Box plots of the reconstruction error of seven ROMs built with different number of POD modes. Each box plot is based on 50 realizations with the vector $(k,\mu_{\rm r},\mu_{\rm i},\xi_{\rm r},\xi_{\rm i})$ taking random values from $[5,10]\times[10,30]\times[10,30]\times[0,100]\times[-100,100]$.}
\label{POD-numMode50pt}
\end{center}
\end{figure}
    
\begin{figure}[!h]
\centerline{
\includegraphics[scale=0.38]{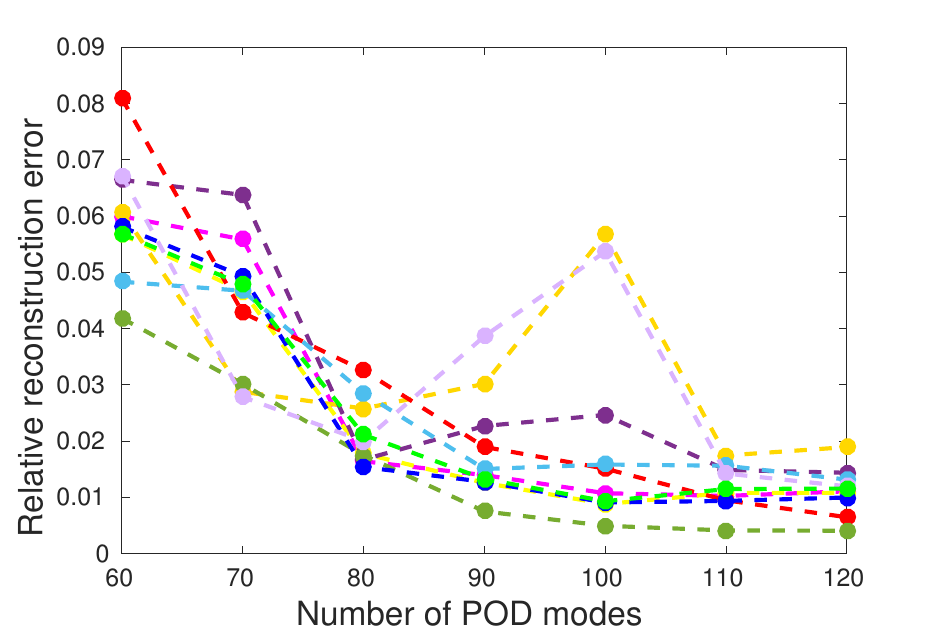}\quad
\includegraphics[scale=0.38]{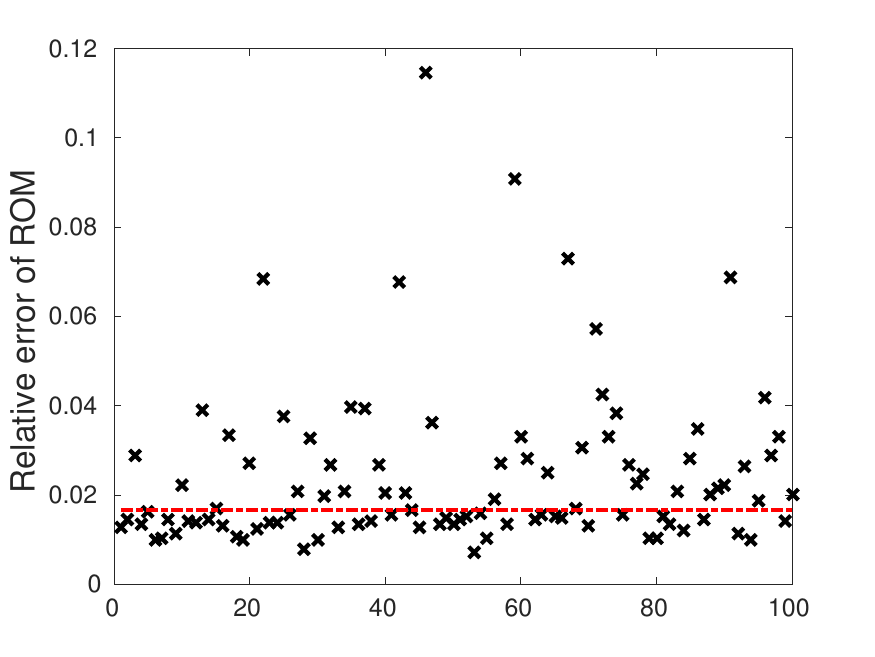} 
}
\caption{Left: ROM accuracy with ten testing parameter values fixed. Right: relative error (black x) of the 90-mode ROM for 100 random realizations. The red dash line denotes the median.}
\label{POD-numMode-fixed}
\end{figure}

In Table \ref{romTimeCmpTable}, the computational efficiency of the ROM relative to the FOM is observed by running simulations on 16 processors. We assume that there are a total of $n_pq$ realizations of the Helmholtz equation to be performed online, where $n_p$ denotes the number of processors. With the ROM, the number of jobs distributed to each processor would be $q$ rather than $n_pq$ as compared with the FOM. The reason is that the ROM keeps only serial dense matrices ($\mathbf{M}_{\rm r}, \mathbf{S}_{\rm r}, \mathbf{K}_{\rm 2r}, \widetilde{\mathbf{K}}_{\rm 2r}, \widetilde{\mathbf{K}}_{\rm 4r}, \mathbf{I}_{\rm r}$) which are of size 90 and can be stored on each processor. 
As such, $n_p$ reduced-order Helmholtz equations can be solved simultaneously and individually. 
For a stochastic Helmholtz solution, the number of realizations $n_pq$ needs to be at least one thousand. In such case, the ROM has a significant gain on efficiency no matter if one counts the total execution time or just the online simulation time because 
$79.15+0.04q \ll 17,750+79.15+0.04q \ll 27.12+50.71n_{\rm p}q$ 
when $q$ is in the thousands or larger.
Note that solving a reduced-order system is $914.3n_p$ faster than computing a full-order system.
The online computational time of the ROM can be further reduced if we project the reduced basis offline, when there is no need to build the finite element space online (e.g., for visualization).

\begin{table}[h!]
\begin{center}
\caption{Comparison of execution time for $n_pq$ realizations of the full-order model (FOM) and the reduced-order model (ROM), including solving their adjoint equations and the sensitivity w.r.t.~the impedance parameter. Here $n_p$ denotes the number of processors used for the simulations. The time displayed below corresponds to $n_p = 16$ in particular.  }
\begin{small}
\begin{tabular}{cllclc}
\toprule
& & \multicolumn{2}{c|}{FOM} & \multicolumn{2}{c}{ROM} \\
     \cmidrule(l){3-6} 
& & CPU time (s) & \# exec./proc & CPU time (s) & \# exec./proc \\
\midrule
offline & construct POD basis & --- & --- & 17,750 & 1   \\
\midrule
\multirow{7}{*}{online}
& load mesh & 24.13 & 1 & 24.13 & 1 \\
& build FE matrices & 2.99 & 1 & 2.99 & 1 \\
& import basis & --- & --- & 47.41 & 1 \\
& project basis & --- & --- & 4.62 & 1 \\
& assemble \& solve state & 22.49 & $n_{\rm p}q$ & 0.03 & $q$ \\
& assemble \& solve adj. & 28.22 & $n_{\rm p}q$ & 0.01 & $q$ \\
& compute sensitivity & 0.10 & $n_{\rm p}q$ & 6.7e-05 & $q$ \\
\midrule
online & total & 27.12+50.71$n_{\rm p}q$ & & 79.15+0.04$q$& \\
\bottomrule
\end{tabular}
\end{small}
\label{romTimeCmpTable}
\end{center}
\end{table}

\subsection{Impedance stochastic optimization using the CVaR measure}\label{sec:resultCVaR}

In this section we apply the ROM to impedance stochastic optimization based on the CVaR measure. First, we verify the accuracy and efficiency of the ROM applied to the optimization setting.
Because the full-order Helmholtz solver is computationally much too expensive for  stochastic problems, we instead use deterministic impedance optimization problems for this validation. Specifically, we consider
the deterministic full-order problem 
$\min\limits_{\xi} \frac{1}{2}\mathbf{p}(\xi;\vartheta^o)^T\mathbf{M}\mathbf{p}(\xi;\vartheta^o)$ 
and compare with the deterministic reduced-order problem 
$\min\limits_{\xi} \frac{1}{2}\Vert \mathbf{p}_{\rm rb}(\xi;\vartheta^o,\mathbb{Z})\Vert_{2}^2 $. 
for a fixed $\vartheta^o$ is fixed, taking the value $(k, \mu_{\rm r}, \mu_{\rm i}) = (10,30,30)$. The performance of applying the reduced basis $\mathbb{Z}$ is shown in Figure~\ref{compFOMandROMitr}.
Starting from the initial guess $\xi=10+10i$, the optimization iterations on the FOM and ROM converge to similar optimal values: $\xi_*=1.154-1.425i$ versus $\xi_*=1.141-1.412i$.  With the same stopping criteria, the full-order optimization use 15 iterations and a total of 23 evaluations of the state problem, compared to 13 iterations and 20 state evaluations for the  reduced-order optimization.

\begin{figure}[h!]
\begin{center}
\includegraphics[scale=0.45]{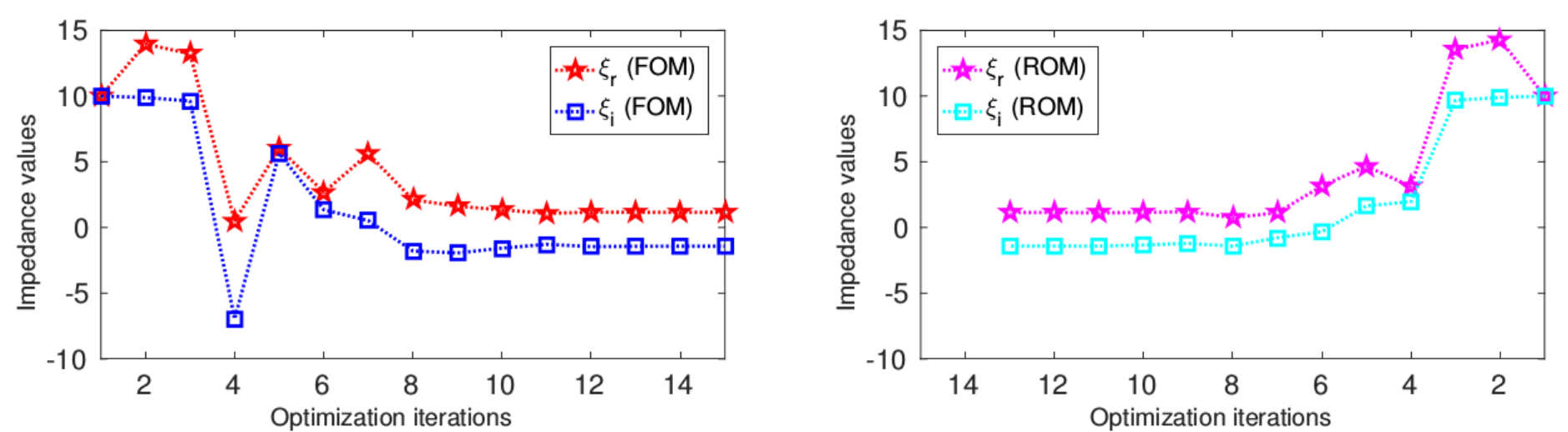}
\caption{Comparison of a deterministic impedance optimization with the FOM and ROM, showing that the ROM can be effectively used for optimization problems.}
\label{compFOMandROMitr}
\end{center}
\end{figure}

We now focus only on the reduced-order stochastic optimization problem (\ref{Jr-discr}) for which the parameter $\varepsilon$ is chosen as $10^{-4}$. We take the value of $\Vert \mathbf{p_{\rm rb}}(\xi^\infty;\vartheta^o)\Vert^2_{2} = 86,588,500$ as the coefficient $\gamma_{\rm p}$ to scale the energy of the acoustic potential in (\ref{Jr-discr}), where $\xi^\infty$ is the infinity value that corresponds to a hard-wall condition imposed on $\Gamma_2$. We  solve the reduced-order stochastic Helmholtz equation at 16,000 Monte Carlo samples distributed equally on 16 processors, i.e., each processor handles 1,000 realizations of the Helmholtz equation for each evaluation of $J^Q_{\rm r}$ and its gradient.

We present in Table \ref{optControlTableQ16k} the computational results for the probability levels $\beta \in \{0.5, 0.75, 0.95\}$. The BFGS iteration starts at $\xi^1 = 10+10i$ and stops when one of the following stopping criterion is satisfied: the maximum iteration number (100) is exceeded, the relative reduction of $\nabla J^Q_{\rm r}$ or $J^Q_{\rm r}$ is more than $10^{-6}$ (i.e.~$|\nabla J^Q_{\rm r}(\xi^{k+1})|\leq 10^{-6}|\nabla J^Q_{\rm r}(\xi^1)|$ or $|J^Q_{\rm r}(\xi^{k+1})|\leq 10^{-6}|J^Q_{\rm r}(\xi^1)|$), the relative step change is less than $10^{-6}$ (i.e.~$|\xi^{k+1}-\xi^k|\leq10^{-6}|\xi^k|$). For all test cases, the optimization solver converges within 30 iterations. Even though a large number of PDEs (more than 36,000) are solved on each processor, the procedure takes at most 487.9 seconds. A detailed plot of the BFGS iteration history corresponding to $\beta=0.95$ is shown in Figure~\ref{romItrFunbeta095}. In the plot, the control variable $\alpha$ and the cost function $J^Q_{\rm r}$ are scaled by 10 for better visualization.


\begin{table}[h!]
\begin{center}
\caption{Impedance optimization with $Q=$ 16,000 Monte Carlo samples equally distributed on 16 processors. }\label{optControlTableQ16k}
\begin{small}
\begin{tabular}{llll}
\toprule
measure & CVaR ($\beta=0.5$)  & CVaR ($\beta=0.75$) & CVaR ($\beta=0.95$)\\
\midrule
optimal impedance & $0.8576-1.2i $    & $0.8893-1.218i  $ & $0.9752-1.267i$ \\
optimal $\alpha$ & 0.2787  & 0.3588 & 0.4866\\
final $J_{\rm r}^Q$ & 0.1889  & 0.2183 & 0.2660\\
\# BFGS iters & 16  & 17  & 27\\
\# PDE solves/proc & 40,000 &  36,000 & 62,000 \\
online exec.~time & 311.2 sec  &  280.9 sec & 487.9 sec\\
\bottomrule
\end{tabular}
\end{small}
\end{center}
\end{table}

\begin{figure}[h!]
\begin{center}
\includegraphics[scale=0.5]{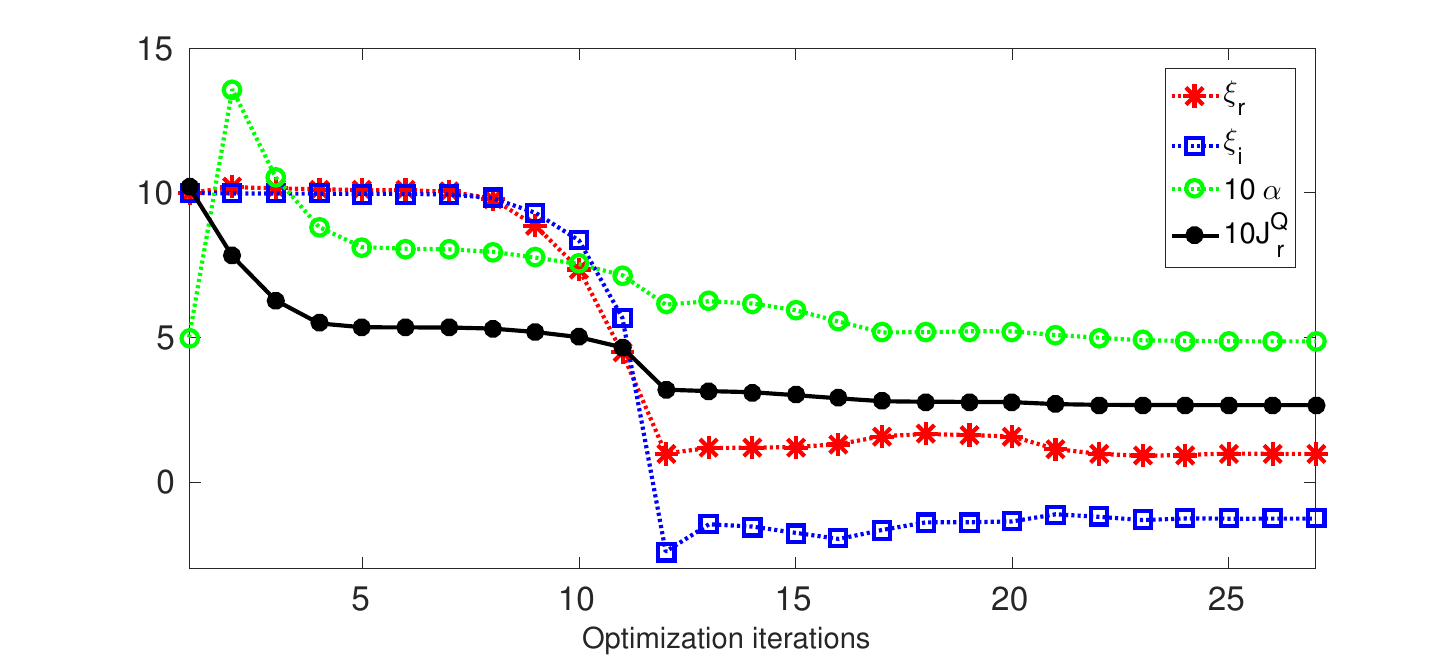} 
\caption{BFGS iterations of the optimization of the CVaR measure with $\beta=0.95$. The control variable $\alpha$ and the cost function $J^Q_{\rm r}$ are scaled by 10 in this plot for better visualization.}
\label{romItrFunbeta095}
\end{center}
\end{figure}

The optimal values listed in Table \ref{optControlTableQ16k}
 provide good suggestions for the acoustic liner design for different significance levels. For instance, with the impedance value $\xi_*=0.8576-1.2i$ we are 50\% sure that the acoustic pressure energy $\Vert\mathbf{p}_{\rm rb}\Vert^2_{2}$ will not exceed $0.2787\gamma_p$ (roughly speaking, this 50\%-threshold is the $\text{VaR}_{0.5}$ value of the energy associated with $\xi_*$). This impedance is optimal in reducing the mean of the acoustic pressure energies above those 50\%-thresholds.
It is also interesting to see that the optimal impedance $\xi_*$ varies slightly with the probability level $\beta$: the higher the level, the larger the values of $|\xi_{\rm r}|$ and $|\xi_{\rm i}|$.

As can be seen from Table \ref{optControlTableQ16k}, we are 95\% sure that the acoustic noise energy can be optimally controlled within 48.66\% of $\gamma_p$, which measures the noise level associated with the hard-wall condition. To have a further indication of the extent of fan noise reduction, we illustrate the mean and standard deviation of the spatial noise energy function $\frac{n}{\gamma_p} |p|^2(\mathbf{x})$ in Figure~\ref{pnorm2onSmpMeanSD} (left and right respectively). Here $n$ denotes the number of degrees of freedom of the full-order model. The first row corresponds to the initial guess $\xi = 10+10i$ whereas the second row corresponds to the optimal impedance $\xi=0.9752-1.267i$. A slice perpendicular to the spinner axis is added in each picture for three-dimensional visualization. As desired, the fan noise level is significantly reduced  when an optimal impedance parameter is taken. We also show, in Figure~\ref{pRek10s1010}, the noise distribution (real part of the pressure $p$) associated with different impedance values, fixing the random parameter $(k,\mu_{\rm r},\mu_{\rm i}) = (10,10,10)$.  The noise is mostly confined near the fan inlet (second row of Figure~\ref{pRek10s1010}) when an optimal impedance value is taken whereas it propagates to the far-field area (first row of Figure~\ref{pRek10s1010}) using the initial impedance value.

\begin{figure}[h!]
\centerline{
\includegraphics[scale=0.18]{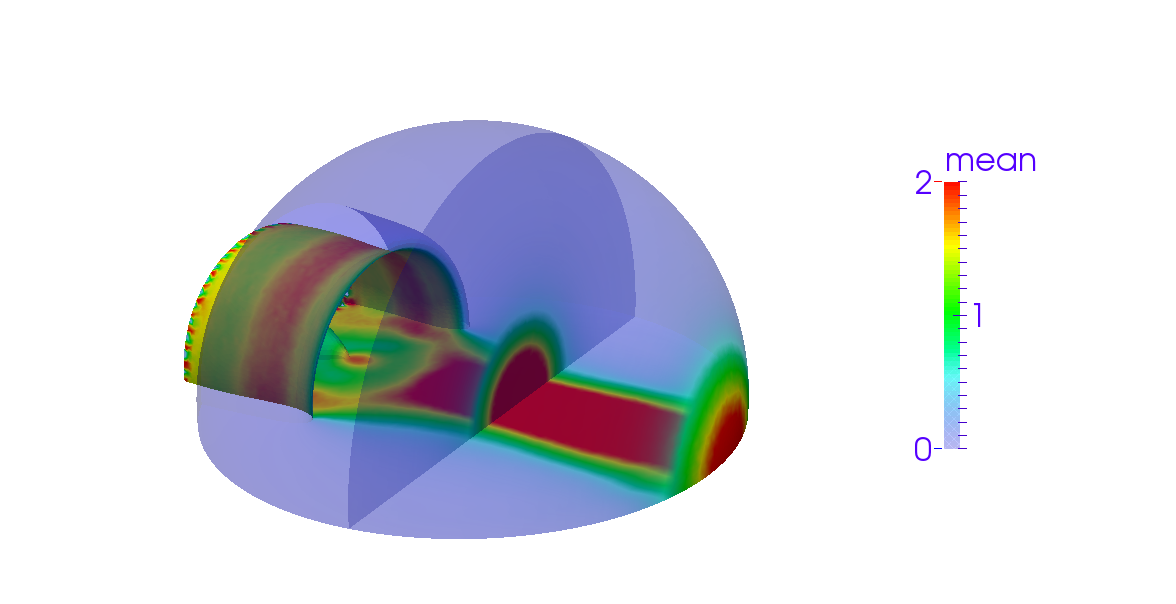}
\includegraphics[scale=0.18]{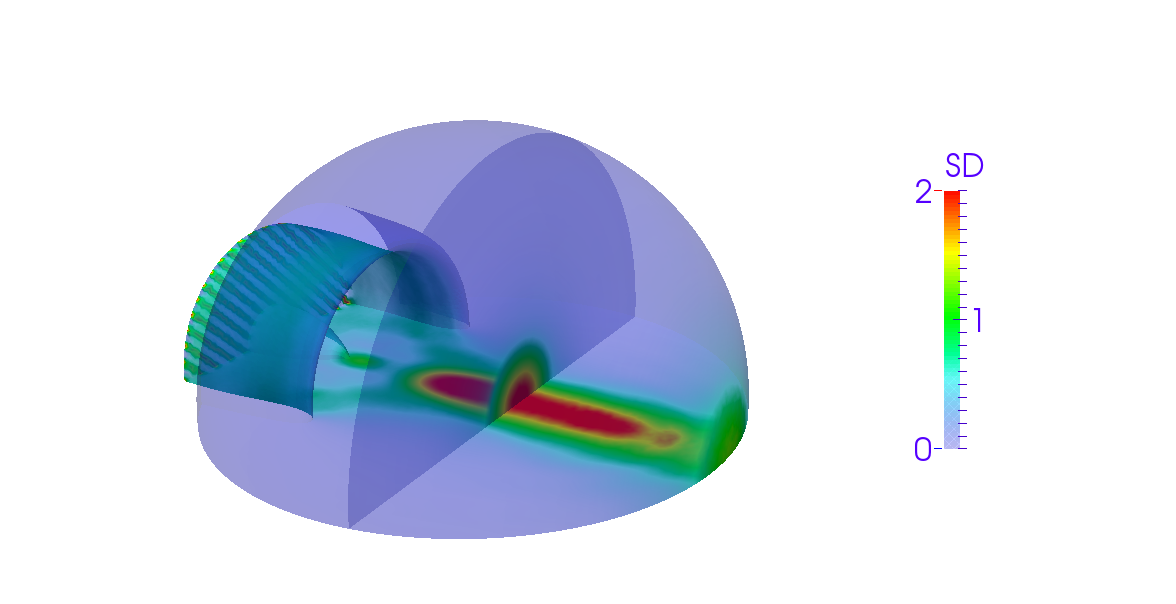}
}
\centerline{
\includegraphics[scale=0.18]{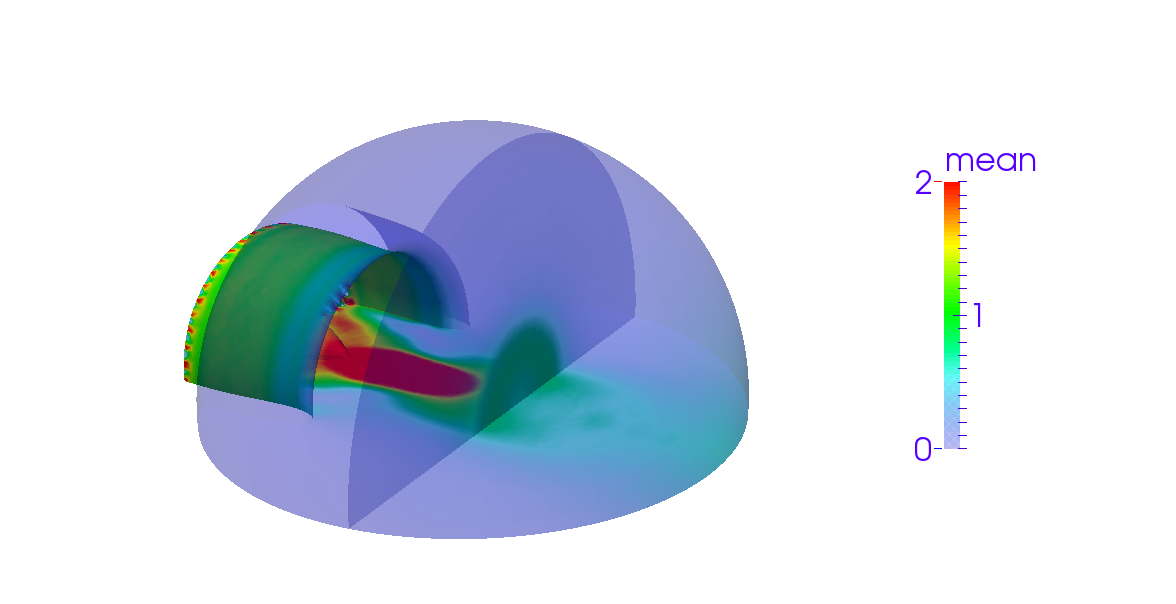}
\includegraphics[scale=0.18]{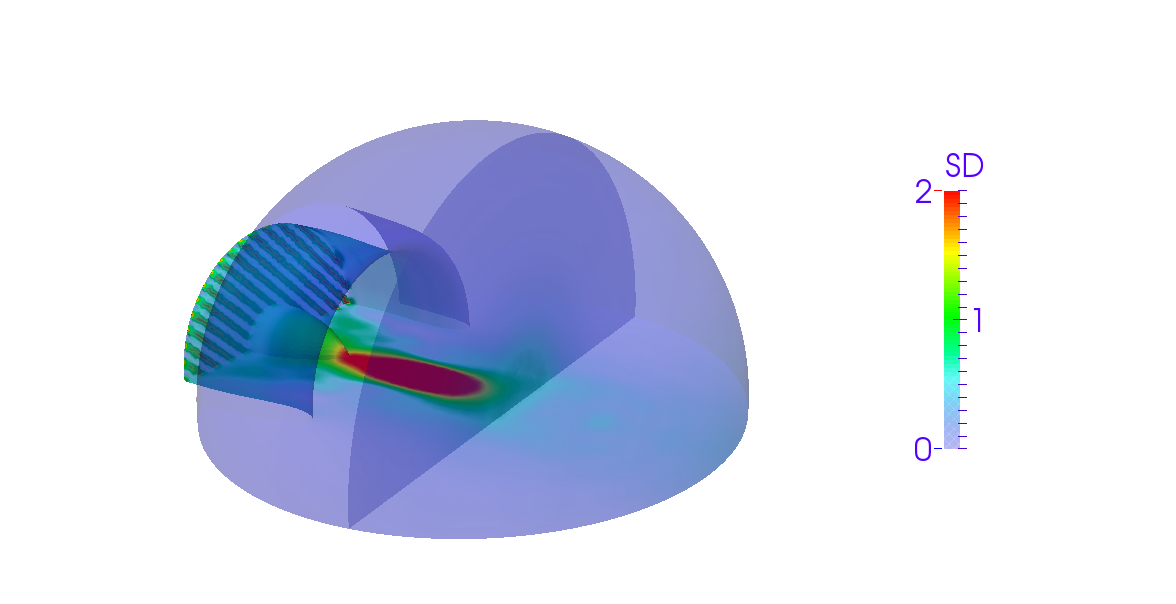}
}
\caption{The means and standard deviations of the spatial noise energy function $\frac{n}{\gamma_p} |p|^2(\mathbf{x})$ corresponding to different impedance values. Here $n$ is the number of degrees of freedom of the FOM. First row: for the initial guess of the impedance $\xi=10+10i$; second row: for the optimal impedance $\xi=0.9752-1.267i$.}
\label{pnorm2onSmpMeanSD}
\end{figure}

\begin{figure}[h!]
\centerline{
\includegraphics[scale=0.18]{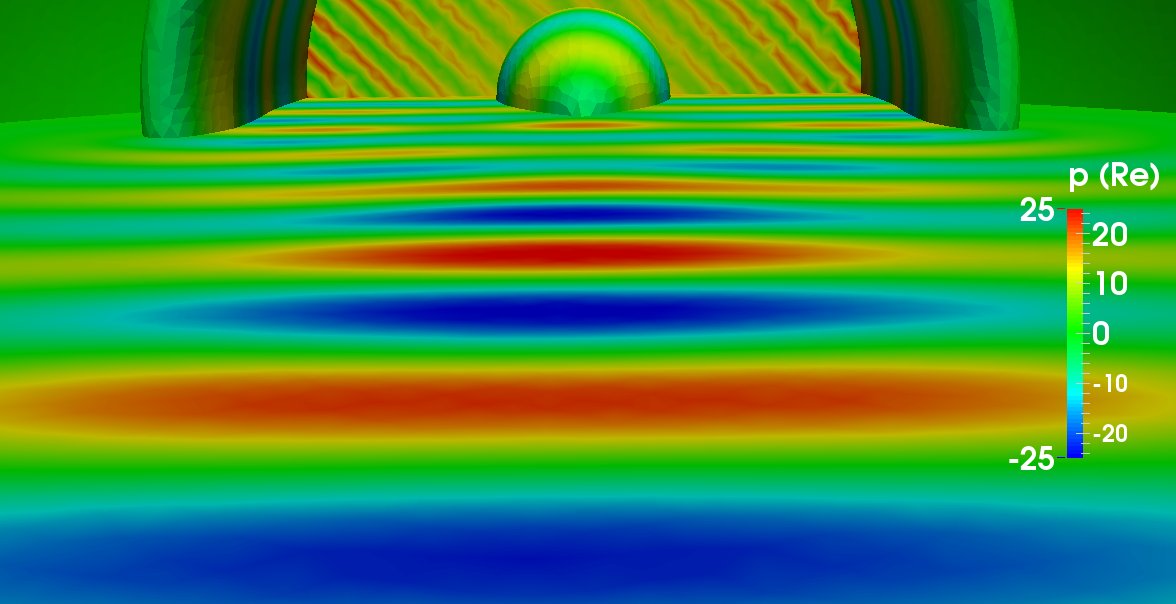}
\quad
\includegraphics[scale=0.18]{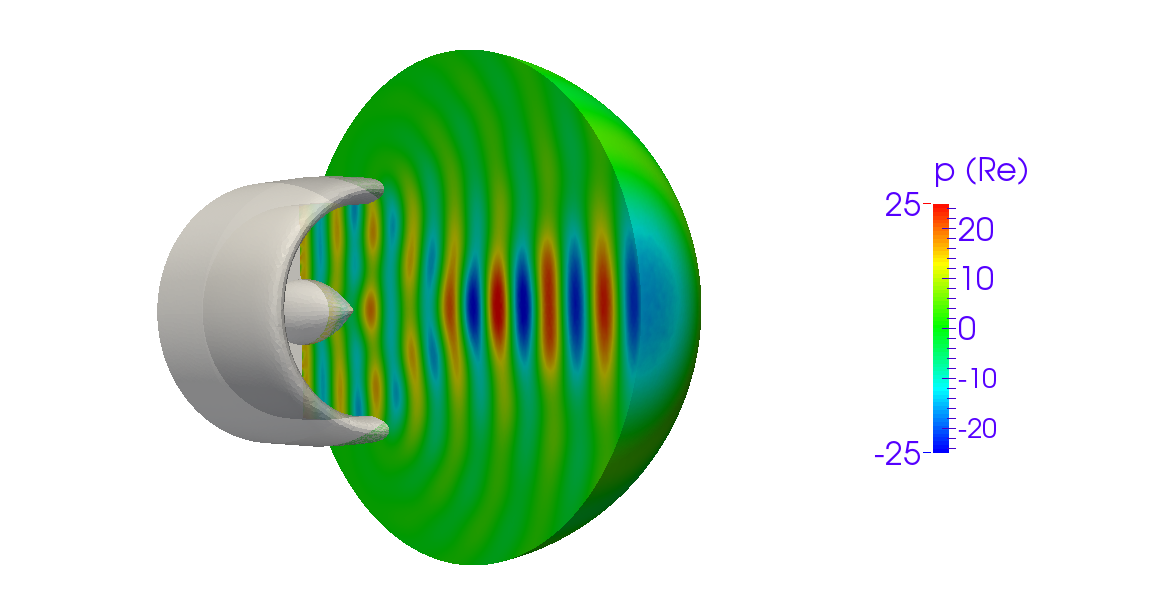}
}
\centerline{
\includegraphics[scale=0.18]{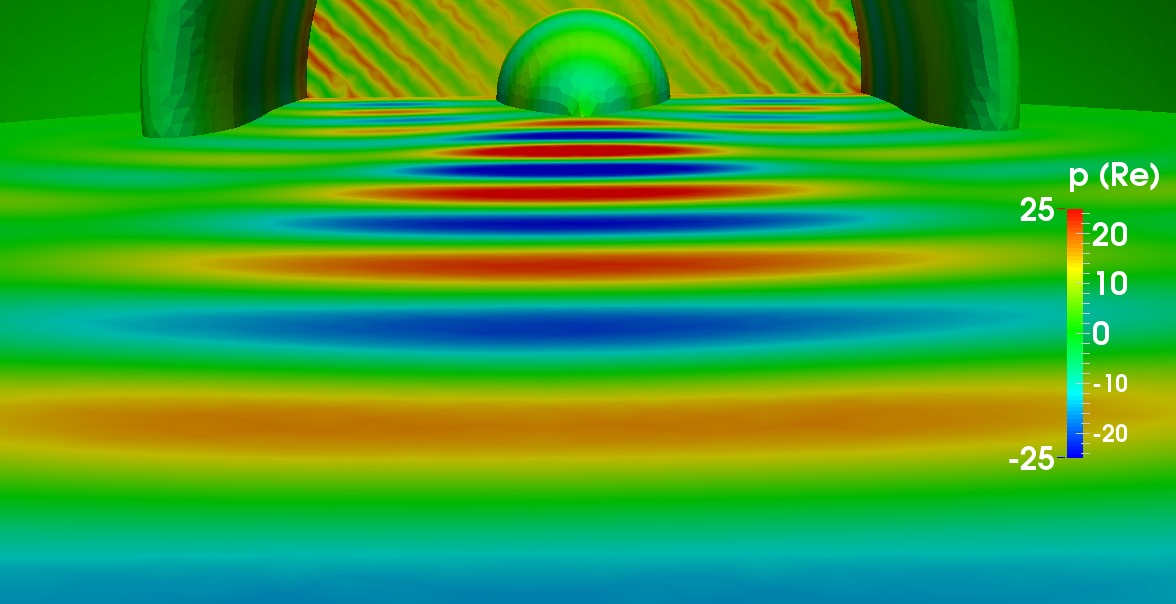}
\quad
\includegraphics[scale=0.18]{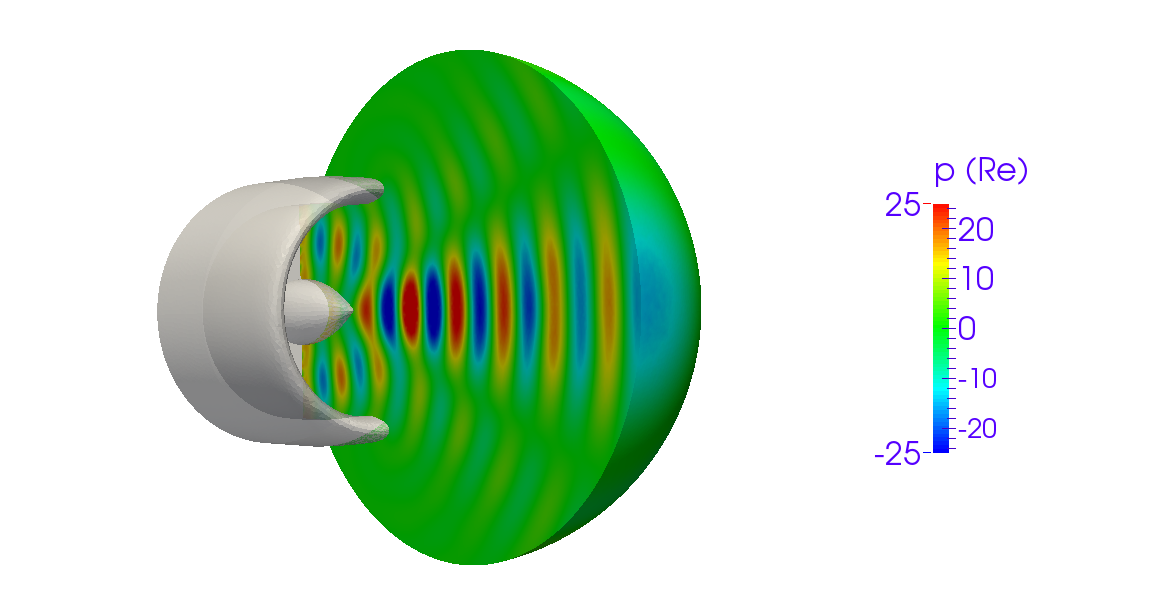}
}
\caption{The real part of the acoustic pressure computed with different impedance values fixing $(k,\mu_{\rm r},\mu_{\rm i}) = (10,10,10)$. First row: for the initial guess of impedance $\xi=10+10i$; second row: for the optimal impedance $\xi=0.9752-1.267i$.}
\label{pRek10s1010}
\end{figure}

\section{Mathematical and numerical analyses}\label{analysis}

In this section, we provide mathematical analyses of the Helmholtz equation and the corresponding optimization problem. 

Recall the previously defined spaces: $V_0 = H^1_{\Gamma_1}(D; \mathbb{C}) = \{\phi\in H^1(D; \mathbb{C}): \phi|_{\Gamma 1} = 0 \}$, 
$V_{\vartheta} = \{p\in H^1(D; \mathbb{C}): p|_{\Gamma 1} = \mu g_{\Gamma_1} \}$,  
$Y_0 = L^2_{\rho}(\Lambda; V_0)$,
and $Y = \{p(\cdot,\vartheta): \Lambda \to V_{\vartheta},  \int_{\Lambda} \Vert p(\cdot,\vartheta)\Vert_{V_{\vartheta}}^2\rho(\vartheta)d\vartheta < \infty \}$.
Norms or semi-norms on a geometric domain $\Sigma$ are denoted by
$\norm{u}_{k,\Sigma} = \norm{u}_{H^k(\Sigma;\mathbb{C})}$ and $|u|_{k, \Sigma}=\norm{D^k u}_{0,\Sigma}$.
We further introduce the $k$-dependent norm on $H^1(\Sigma;\mathbb{C}): \norm{u}_{\mathcal{H},\Sigma} = k\norm{u}_{0, \Sigma} + |u|_{1, \Sigma}$. Throughout the analysis, we frequently use the notation $C$, with or without subscripts, to denote a generic positive constant or continuous function.

\subsection{Well-posedness analysis}\label{wellposeAna}
Assume $g_{\Gamma_1}(\mathbf{x}) \in H^{1/2}(\Gamma_1)$ and that $p_g(\mathbf{x}) \in H^1(D)$ is the unique solution of
\begin{equation}\label{PDELapla}
\left\{
\begin{array}{ll}
-\Delta p_g(\mathbf{x}) = 0 &  \mbox{in } D \\[0.2cm]
p_g(\mathbf{x}) = g_{\Gamma_1}(\mathbf{x}) &  \mbox{on } \Gamma_1 \\[0.2cm]
\frac{\partial p_g(\mathbf{x})}{\partial \mathbf{n}} = 0 & \mbox{on } \partial D\backslash\Gamma_1 
\end{array}
\right.
\end{equation}
that is the limit case of (\ref{PDEstate}) as $k\to 0$, omitting the random parameter in the Dirichlet boundary condition.
Assuming $p$ solves the Helmholtz equation (\ref{PDEstate}), the lifted solution $\widetilde{p} = p - \mu p_g$ then satisfies
\begin{equation}\label{PDEstateLift}
\left\{
\begin{array}{ll}
-\Delta \widetilde{p}(\mathbf{x},\vartheta) - k^2 \widetilde{p}(\mathbf{x},\vartheta) = \widetilde{f} &  \mbox{in } D \\[0.2cm]
\widetilde{p}(\mathbf{x},\vartheta) = 0 &  \mbox{on } \Gamma_{\rm d} \\[0.2cm]
\frac{\partial \widetilde{p}(\mathbf{x},\vartheta)}{\partial \mathbf{n}} = 0 & \mbox{on }\Gamma_{\rm n}  \\[0.2cm]
\frac{\partial \widetilde{p}(\mathbf{x},\vartheta)}{\partial \mathbf{n}} = (i\widetilde{\beta}-\widetilde{\alpha}) \widetilde{p}(\mathbf{x},\vartheta) + \widetilde{g} & \mbox{on } \Gamma_{\rm r} ,
\end{array}
\right.
\end{equation}
where the Dirichlet, Neumann, and Robin boundaries are given by $\Gamma_{\rm d} = \Gamma_1$, $\Gamma_{\rm n}=\Gamma_3\cup\Gamma_5$, and $\Gamma_{\rm r}=\Gamma_2\cup\Gamma_4$, respectively. The constants and right-hand sides are given by
\begin{equation}\label{PDEstateLiftRHS}
\widetilde{\beta} = \left\{
\begin{array}{ll}
-\frac{\xi_{\rm r}}{|\xi|^2}k	&	\text{on	} \Gamma_2 \\[0.1cm]
-k	&	\text{on	} \Gamma_4
\end{array}
\right.
\qquad
\widetilde{\alpha} = \left\{
\begin{array}{ll}
\frac{\xi_{\rm i}}{|\xi|^2}k	&	\text{on	} \Gamma_2 \\[0.1cm]
0	&	\text{on	} \Gamma_4
\end{array}
\right.
\qquad
\widetilde{f} = \mu k^2p_g 
\qquad
\widetilde{g} = \left\{
\begin{array}{ll}
-i\mu\frac{k}{\xi}p_g	&	\text{on	} \Gamma_2 \\[0.1cm]
-i\mu kp_g	&	\text{on	} \Gamma_4.
\end{array}
\right.
\end{equation}
Note that $\widetilde{\beta}$ and $\widetilde{\alpha}$ satisfy
\begin{equation}\label{bddBetaAlp}
0 < C_{\beta,-}k \leq -\widetilde{\beta} \leq C_{\beta,+}k \quad\mbox{and}\quad |\widetilde{\alpha}|\leq C_{|\alpha|}k \quad \mbox{on}~ \Gamma_{\rm r}
\end{equation}
by setting $C_{\beta,-}=\min\{\frac{\xi_{\rm r}}{|\xi|^2} ,1\}, C_{\beta,+}=\max\{\frac{\xi_{\rm r}}{|\xi|^2} ,1\}$, 
 and $C_{|\alpha}| = \frac{|\xi_{\rm i}|}{|\xi|^2}$.
Because $\Lambda$ is bounded, $\mu p_g $ belongs to $Y$.
A weak formulation to (\ref{PDEstate}) is then given by: find $p = \mu p_g + \widetilde{p}\in \mu p_g + Y_0 = \mu p_g + L^2_\rho(\Lambda; H^1_{\Gamma_1}(D;\mathbb{C}))$ such that
$$
\int_\Lambda a_\vartheta(\widetilde{p},\phi)\rho(\vartheta)d\vartheta = \int_\Lambda b_\vartheta(\phi)\rho(\vartheta)d\vartheta \qquad \forall \phi \in Y_0,
$$
where
$$
b_\vartheta(\phi) = \mu k^2\int_D p_g\overline{\phi}d\mathbf{x}-i\mu\frac{k}{\xi}\int_{\Gamma_2}p_g\overline{\phi}ds-i\mu k\int_{\Gamma_4}p_g\overline{\phi}ds.
$$

The well-posedness of deterministic solutions to (\ref{PDEstate}) is proved in \cite[Theorem 1]{Cao2007Uncertain}, where the unique solvability is provided except for a countable set of $k$. Here, we use a classical approach to show the unique solvability for any values of the random parameter $\vartheta$.

\begin{thm}\label{prop1}
For any $\xi$ with $\xi_{\rm r}>0$ and every $\vartheta = [k, \mu_{\rm r}, \mu_{\rm i}] \in \mathbb{R}^{+}\times \mathbb{R}^2$, there exists a unique weak solution $p(\cdot,\vartheta;\xi)\in V_{\vartheta}$ solving \eqref{PDEstate}. 
\end{thm}

\begin{proof}
We only need to prove the existence and uniqueness of the solution of the problem: find $\widetilde{p}(\cdot,\vartheta)\in H^1_{\Gamma_1}(D;\mathbb{C})$ such that
\begin{equation}\label{liftVarDeter}
a_\vartheta(\widetilde{p}, \phi_D) = b_\vartheta(\phi_D), \qquad \forall \phi_D \in H^1_{\Gamma_1}(D;\mathbb{C}).
\end{equation}
The sesquilinear form $a_\vartheta(\cdot, \cdot): H^1_{\Gamma_1}(D;\mathbb{C}) \times H^1_{\Gamma_1}(D;\mathbb{C}) \to \mathbb{C}$ is continuous. Indeed,
\begin{equation}\label{contiOnHNm}
\begin{aligned}
|a_{\vartheta}(u,v)| \leq& \norm{\nabla u}_{0,D}\norm{\nabla v}_{0,D} + k^2\norm{u}_{0,D}\norm{v}_{0,D} + C(\xi)k\norm{u}_{0,\partial D}\norm{v}_{0,\partial D} \\
 \leq & 2\norm{u}_{\mathcal{H},D}\norm{v}_{\mathcal{H},D} + C(\xi)\big(k^2\norm{u}_{0,D}\norm{\nabla u}_{0,D}\norm{v}_{0,D}\norm{\nabla v}_{0,D}\big)^{1/2}\\
 \leq & 2\norm{u}_{\mathcal{H},D}\norm{v}_{\mathcal{H},D} + C(\xi)\big(k^2\norm{u}_{0,D}\norm{ v}_{0,D}+\norm{\nabla u}_{0,D}\norm{\nabla v}_{0,D}\big)\\
 \leq & 2\norm{u}_{\mathcal{H},D}\norm{v}_{\mathcal{H},D} + C(\xi)\norm{u}_{\mathcal{H},D}\norm{v}_{\mathcal{H},D}
 \leq  C_0(\xi)\norm{u}_{\mathcal{H},D}\norm{v}_{\mathcal{H},D}.
\end{aligned}
\end{equation}
We also observe that $$\Re[a_\vartheta(\widetilde{p}, \widetilde{p})] = |\widetilde{p}|^2_{1,D} - k^2 \norm{\widetilde{p}}^2_{0,D} + \frac{k\xi_{\rm i}}{|\xi|^2}\norm{\widetilde{p}}^2_{0,\Gamma_2} 
\geq \norm{\widetilde{p}}^2_{1,D} - (1+k^2) \norm{\widetilde{p}}^2_{0,D} - \frac{k|\xi_{\rm i}|}{|\xi|^2}\norm{\widetilde{p}}^2_{0,\Gamma_2}.$$
By the trace theorem \cite[Theorem 1.5.1.10]{Grisvard1985}, 
$
\norm{\widetilde{p}}^2_{0,\partial D} \leq C(\epsilon\norm{\nabla \widetilde{p}}^2_{1,D} + \epsilon^{-1}\norm{\widetilde{p}}^2_{0,D})
$
for any $\epsilon\in(0,1)$. Therefore, the G\r{a}rding inequality 
$$\Re[a_\vartheta(\widetilde{p}, \widetilde{p})] \geq (1-\frac{k|\xi_{\rm i}|C}{|\xi|^2}\epsilon)\norm{\widetilde{p}}^2_{1,D} - (1+k^2+\frac{k|\xi_{\rm i}|C}{|\xi|^2\epsilon}) \norm{\widetilde{p}}^2_{0,D}$$
is satisfied by choosing sufficiently small $\epsilon$ such that $1-\frac{k|\xi_{\rm i}|C}{|\xi|^2}\epsilon > 0$.
Consequently, by \cite[Theorems 2.27 and 2.34]{ellipBook2000}, the Fredholm alternative applies to the sesquilinear form $a_{\vartheta}$. That is, to prove the unique solvability of (\ref{liftVarDeter}) it is enough to show  the associated homogeneous problem has only trivial solution.


Assume $q \in H^1_{\Gamma_1}(D;\mathbb{C})$ solves the homogeneous Helmholtz equation corresponding to (\ref{PDEstateLift}) and (\ref{PDEstateLiftRHS}). Its weak formulation implies
$$\Im[a_\vartheta(q, q)] =  \frac{k\xi_{\rm r}}{|\xi|^2}\norm{q}^2_{0,\Gamma_2} + k\norm{q}^2_{0,\Gamma_4} = 0. $$
Hence $q = 0$ a.e.~on $\Gamma_2 \cup \Gamma_4$. Let us extend the domain $D$ near an interior point of $\Gamma_2$ (or $\Gamma_4$), and denote the extended domain as $D_{\rm et}$. Notice that $D_{\rm et} \supset D$ and they share the Dirichlet and Neumann boundaries. The extension
\begin{equation*}
q_{\rm et} = \left\{
\begin{array}{ll}
q	&	\mbox{if } \mathbf{x}\in D \\
0	&	\mbox{if } \mathbf{x}\in D_{\rm et}\backslash D
\end{array}
\right.
\end{equation*}
is also a weak solution in $H^1_{\Gamma_1}(D_{\rm et};\mathbb{C})$ solving the homogeneous Helmholtz equation corresponding to (\ref{PDEstateLift}) and (\ref{PDEstateLiftRHS}) with domain $D$ replaced by $D_{\rm et}$ and $\Gamma_{\rm r}$ replaced by the extended boundary. Because $q_{\rm et}$ vanishes in a sub-domain of $D_{\rm et}$, by the unique continuation principle \cite{Leis1986} it should vanish identically on $D_{\rm et}$. Therefore, $q \equiv 0$ on $D$.
Π\end{proof}

\begin{as}\label{assump}
For any $\xi$ with $\xi_{\rm r}> 0$ and given $\vartheta = [k, \mu_{\rm r}, \mu_{\rm i}] \in \mathbb{R}^{+}\times \mathbb{R}^2$ with bound constraints on $\widetilde{\beta}$ and $\widetilde{\alpha}$ as in \eqref{bddBetaAlp}, the deterministic weak solution $\widetilde{p}$ of \eqref{PDEstateLift} satisfies
\begin{equation}\label{ineqAssum}
\norm{\widetilde{p}}_{\mathcal{H},D}  \leq C_1(k,\xi)(\Vert{\widetilde{f}}\Vert_{0,D} + \norm{\widetilde{g}}_{0,\Gamma_{\rm r}}).
\end{equation}
Moreover, we assume the solution to \eqref{PDEstateLift} and \eqref{PDEstateLiftRHS} satisfies
\begin{equation}\label{ineqAssumThis}
\norm{\widetilde{p}}_{\mathcal{H},D} \leq C_2(\mu,\xi)P_\alpha(k)(\Vert{p_g}\Vert_{0,D} + \norm{p_g}_{0,\Gamma_{\rm r}}).
\end{equation}
Here $C_1(k,\xi)$ and $C_2(\mu,\xi)$ are continuous functions of $\xi$ and $k$ or $\mu$; $P_\alpha(k)$ is a polynomial in $k$ with degree $\alpha$. These coefficient functions depend only on the domain $D$.
\end{as}

\noindent
\textbf{Remark 1}. Under Assumption \ref{assump}, for any $\xi$ with $\xi_{\rm r}>0$, the stochastic Helmholtz equation (\ref{PDEstate}) has a unique solution $p \in Y$. In fact, from the boundedness of $\Lambda$ in $\mathbb{R}^+\times\mathbb{R}^2$ and the continuity of $C_2(\mu,\xi)P_{\alpha}(k)$ in (\ref{ineqAssumThis}), it is straightforward to show that $\int_\Lambda |\widetilde{p}(\cdot,\vartheta)|^2_{1,D}\rho(\vartheta)d\vartheta < \infty.$ A similar argument for $\norm{\widetilde{p}}_{0,D}$, or the Poincar\'{e} inequality, indicates that $\int_\Lambda \norm{\widetilde{p}(.,\vartheta)}^2_{0,D}\rho(\vartheta)d\vartheta < \infty.$ Therefore, $\widetilde{p}\in Y_0$ and hence $p = \mu p_g+\widetilde{p} \in Y$. To see the uniqueness, take test functions in (\ref{weakHelmState}) as $\phi = \phi_{\Lambda}(\vartheta)\phi_{D}(\mathbf{x})$ with $\phi_{\Lambda}(\vartheta) \in L^2(\Lambda)$ and $\phi_{D}(\mathbf{x}) \in H^1_{\Gamma_1}(D;\mathbb{C})$. The uniqueness of a stochastic solution is then reduced to the uniqueness of deterministic solutions, which are guaranteed by Theorem ~\ref{prop1}.
\vspace{0.2cm}

\noindent
\textbf{Remark 2}. Stability estimates for the Helmholtz equation as in Assumption \ref{assump} have been studied in many  papers, but mainly with a Robin boundary condition on a star-shaped or convex domain \cite{MelenkThesis1995, CUMMINGS2006, Moiola2014}. Therein ``star-shaped" generally means a condition as stated in the third expression of (\ref{domainCond}). The best bound for Helmholtz solutions in terms of data has been given by \cite{Spence2014} on a bounded Lipschitz domain, but does not apply to the case with mixed boundary conditions. In the following, we state that Assumption \ref{assump} holds at least under some constraints on the geometric domain, mainly referring to \cite{hetmaniuk2007}.

\begin{prop}\label{propStarshp} Assumption {\rm\ref{assump}} holds for a domain $D$ with the following constraints:
\begin{itemize}
\item the unique solution $\widetilde{p}$ of \eqref{PDEstateLift} belongs to $H^{3/2+\epsilon}(D)$ with $\epsilon > 0$;
\item there exists a point $\mathbf{x}_0\in \mathbb{R}^3$ and a constant $\gamma_D>0$ such that
\begin{equation}\label{domainCond}
\begin{array}{ll}
(\mathbf{x}-\mathbf{x}_0)\cdot\mathbf{n}(\mathbf{x}) \leq 0 & \forall \mathbf{x}\in\Gamma_{\rm d} \\[0.1cm]
(\mathbf{x}-\mathbf{x}_0)\cdot\mathbf{n}(\mathbf{x}) = 0 & \forall \mathbf{x}\in\Gamma_{\rm n} \\[0.1cm]
(\mathbf{x}-\mathbf{x}_0)\cdot\mathbf{n}(\mathbf{x}) \geq \gamma_D & \forall \mathbf{x}\in\Gamma_{\rm r} .
\end{array}
\end{equation}
\end{itemize}
In such case, the continuous function $C_1(k,\xi)$ in \eqref{ineqAssum} takes the form $C(\xi)(1+\frac{1}{k})$ and the polynomial $P_{\alpha}(k)$ in \eqref{ineqAssumThis} takes the form $k^2+k+1$. Furthermore, the solution $\widetilde{p}$ of \eqref{PDEstateLift} belongs to $H^2(D;\mathbb{C})$ and satisfies
\begin{equation}\label{H2normbdd}
|\widetilde{p}|_{2,D} \leq C_3(\xi)(k+1)(\Vert{\widetilde{f}}\Vert_{0,D} + \norm{\widetilde{g}}_{0,\Gamma_{\rm r}}).
\end{equation}
\end{prop}

\begin{proof}
This result is a straightforward variant from \cite[Propositions 3.3 and 3.4]{hetmaniuk2007}.
To study the case for small wave numbers, we consider the Poisson equation
\begin{equation}\label{poissonKSmall}
\left\{
\begin{array}{ll}
-\Delta \widetilde{h}  = \widetilde{f} &  \mbox{in } D \\[0.1cm]
\widetilde{h} = 0 &  \mbox{on } \Gamma_{\rm d} \\[0.1cm]
\frac{\partial \widetilde{h}}{\partial \mathbf{n}} = 0 & \mbox{on }\Gamma_{\rm n}  \\[0.1cm]
\frac{\partial \widetilde{h}}{\partial \mathbf{n}} = \widetilde{g} & \mbox{on }\Gamma_{\rm r} 
\end{array}
\right.
\end{equation}
which is well posed. Applying the Banach--Ne\v{c}as--Babu\v{s}ka theorem \cite[Theorem 2.6]{ern2010theory}, we have
\begin{equation*}
\Vert{\widetilde{h}}\Vert_{1,D} \leq C(\Vert{\widetilde{f}}\Vert_{0,D}+\norm{\widetilde{g}}_{0,\Gamma_{\rm r}}).
\end{equation*}
The solution $\widetilde{p}$ of (\ref{PDEstateLift}) satisfies (\ref{poissonKSmall}) if we replace the non-homogeneous right hand sides  $\widetilde{f}$ and $\widetilde{g}$ by $k^2\widetilde{p}+\widetilde{f}$ and $(i\widetilde{\beta}-\widetilde{\alpha})\widetilde{p}+\widetilde{g}$, respectively. Therefore,
\begin{eqnarray*}
\norm{\widetilde{p}}_{1,D} &\leq &C(\norm{k^2\widetilde{p}+\widetilde{f}}_{0,D}+\norm{(i\widetilde{\beta}-\widetilde{\alpha})\widetilde{p}+\widetilde{g}}_{0,\Gamma_{\rm r}})\\
& \leq & Ck^2\norm{\widetilde{p}}_{1,D}+C\Vert\widetilde{f}\Vert_{0,D}+ C(C_{\beta,+}k+C_{|\alpha|}k)\norm{\widetilde{p}}_{1,D}+C\norm{\widetilde{g}}_{0,\Gamma_{\rm r}} ,
\end{eqnarray*}
where the third term follows from the trace theorem. 
When $k$ is sufficiently small, say $k < k_0$,
inequality (\ref{ineqAssum}) holds as 
\begin{equation}\label{kSmallIneq}
\norm{\widetilde{p}}_{\mathcal{H},D}  < \norm{\widetilde{p}}_{1,D} \leq C(\Vert{\widetilde{f}}\Vert_{0,D} + \norm{\widetilde{g}}_{0,\Gamma_{\rm r}}).
\end{equation}
When $k\geq k_0$, we resort to \cite[Proposition~3.3]{hetmaniuk2007}, where the assumption on positive $\widetilde{\beta}$ can be changed to negative without effect on the result. When we replace the bound constraint $|\widetilde{\alpha}|\leq C_{|\alpha|}$ in \cite{hetmaniuk2007} by $|\widetilde{\alpha}|\leq C_{|\alpha|}k$, we should have
\begin{equation}\label{kLargeIneq}
\norm{\widetilde{p}}_{\mathcal{H},D} \leq C(\xi)(1+\frac{1}{k})(\Vert\widetilde{f}\Vert_{0,D}+\norm{\widetilde{g}}_{0,\Gamma_{\rm r}}),
\end{equation}
where $C(\xi)$ is continuously dependent on $C_{\beta,-}, C_{\beta,+}, C_{|\alpha|}$, and hence on $\xi$.

Overall, we have (\ref{ineqAssum}) satisfied for any $k>0$ with $C_1(k,\xi)$ in the form $C(\xi)(1+\frac{1}{k})$. If the right hand sides of (\ref{PDEstateLift}) are given as (\ref{PDEstateLiftRHS}), we have
$$
\begin{aligned}
\norm{\widetilde{p}}_{\mathcal{H},D} \leq& C(\xi)(1+\frac{1}{k})(\Vert\widetilde{f}\Vert_{0,D}+\norm{\widetilde{g}}_{0,\Gamma_{\rm r}}) \\
 \leq & C(\xi)(1+\frac{1}{k})(|\mu|k^2\norm{p_g}_{0,D}+ |\mu|(1+\frac{1}{|\xi|})k\norm{p_g}_{0,\Gamma_{\rm r}} )\\
 \leq & C_2(\mu,\xi)(k^2+k+1)(\norm{p_g}_{0,D} + \norm{p_g}_{0,\Gamma_{\rm r}}).
\end{aligned}
$$
We then obtain the $H^2$ estimate as follows:
$$
\begin{aligned}
|\widetilde{p}|_{2,D}  \leq & C\bigg(\norm{\Delta \widetilde{p}}_{0,D} + \norm{\frac{\partial \widetilde{p}}{\partial \mathbf{n}}}_{0,\partial D}\bigg) \\
  = & C\bigg(\norm{k^2\widetilde{p}+\widetilde{f}}_{0,D} + \norm{(i\widetilde{\beta}-\widetilde{\alpha})\widetilde{p}+\widetilde{g}}_{0,\Gamma_{\rm r}}\bigg) \\
 \leq & C\bigg(k^2\norm{\widetilde{p}}_{0,D}+\Vert\widetilde{f}\Vert_{0,D} + C(\xi)k|\widetilde{p}|_{1,D} +\norm{\widetilde{g}}_{0,\Gamma_{\rm r}}\bigg) \\
  \leq & C(\xi)k\norm{\widetilde{p}}_{\mathcal{H},D}+C(\Vert\widetilde{f}\Vert_{0,D} + \Vert\widetilde{g}\Vert_{0,\Gamma_{\rm r}}) \\
  \leq & C_3(\xi)(k+1)(\Vert{\widetilde{f}}\Vert_{0,D} + \norm{\widetilde{g}}_{0,\Gamma_{\rm r}}).
\end{aligned}
$$
\end{proof}

\begin{lem}\label{diffErrXi}
Let $p_1(\mathbf{x},\vartheta)$ and $p_2(\mathbf{x},\vartheta)$ denote two weak solutions of \eqref{PDEstate} in $Y$ with control variables $\xi_1$ and $\xi_2$, respectively. Then, for a.e.~$\vartheta\in\Lambda$, the difference $p_1-p_2$ satisfies (under Assumption {\rm\ref{assump}})
\begin{equation}\label{diffPineq}
\norm{p_1-p_2}_{\mathcal{H},D}  \leq C_1(k,\xi_1)(\norm{p_2}_{0,\Gamma_2} + |\mu|k\norm{p_g}_{0,\Gamma_2})\bigg|\frac{1}{\xi_1}-\frac{1}{\xi_2}\bigg|.
\end{equation}
\end{lem}
\begin{proof}
For a.e.~$\vartheta\in\Lambda$,  $p_1-p_2 \in H^1_{\Gamma_1}(D;\mathbb{C})$ is a weak solution of the problem:
\begin{equation*}
\left\{
\begin{array}{ll}
-\Delta \widehat{p} - k^2\widehat{p}  = 0 &  \mbox{in } D \\[0.1cm]
\widehat{p} = 0 &  \mbox{on } \Gamma_1 \\[0.1cm]
\frac{\partial \widehat{p}}{\partial \mathbf{n}} = 0 & \mbox{on }\Gamma_3\cup\Gamma_5  \\[0.1cm]
\frac{\partial \widehat{p}}{\partial \mathbf{n}} + i\frac{k}{\xi_1}\widehat{p} = 
(\frac{1}{\xi_1}-\frac{1}{\xi_2})(-p_2-i\mu kp_g)
& \mbox{on }\Gamma_2  \\[0.1cm]
\frac{\partial \widehat{p}}{\partial \mathbf{n}} + ik\widehat{p} = 0  & \mbox{on }\Gamma_4 .
\end{array}
\right.
\end{equation*}
Then, (\ref{diffPineq}) follows from Assumption \ref{assump}.
\end{proof}

\begin{thm}\label{existOpt}
For any $\gamma \geq 0$, there exists a solution to the optimization problem \eqref{Jbetaalpha}. 
\end{thm}
\begin{proof}
We verify the conditions assumed in \cite[Assumption 2.2]{Kouri2016cvar}. Given a convergent sequence of the control variable $\xi_n\to\xi_*$ as $n\to \infty$, we denote the corresponding solutions $p(\mathbf{x},\vartheta;\xi_n)$ by $p_n$ and $p(\mathbf{x},\vartheta;\xi_*)$ by $p_*$. From Lemma \ref{diffErrXi}, for a.e.~$\vartheta\in \Lambda$ we have
$$
\norm{p_n-p_*}_{\mathcal{H},D}  \leq C_1(k,\xi_n)(\norm{p_*}_{0,\Gamma_2} + |\mu|k\norm{p_g}_{0,\Gamma_2})\bigg|\frac{1}{\xi_n}-\frac{1}{\xi_*}\bigg|.
$$
Because $C_1(\cdot,\cdot)$ is continuous, letting $n\to\infty$ we conclude that $p_n(\cdot,\vartheta)\to p_*(\cdot,\vartheta)$ in $H^1(D,\mathbb{C})$. By the argument of \cite[Theorem 4.1]{Kouri2016cvar}, there exists a solution to the optimization problem (\ref{Jbetaalpha}).
\end{proof}

\subsection{Numerical analysis}
It is known that finite element methods for the Helmholtz equation are quasi-optimal when the mesh size $h$ is small enough. To be more precise about the size of $h$, we show under the constraints in Proposition~\ref{assump} that a sufficient condition for quasi-optimality is $k^2h \ll 1$. The proof is an extension of Melen's work with a Robin boundary condition \cite[Proposition~8.2.7]{MelenkThesis1995}.

Assume we have a quasi-uniform mesh such that the linear finite element best approximation error is given by
\begin{equation}\label{FEappAss}
|u-\Pi_h u|_{s,D} \leq Ch^{2-s}|u|_{2,D} \quad \forall u\in H^2(D;\mathbb{C}) \quad (s=0,1).
\end{equation}
Here $\Pi_h: L^2 \to V_h$ is the projection into the linear finite element space $V_h$, so $\norm{u-\Pi_h u}_{0,D} = \inf\limits_{v_h\in V_h}\norm{u-v_h}_{0,D}$.

\begin{prop}\label{quasiOpt}
With the same condition as in Proposition~{\rm\ref{propStarshp}}, the finite element solution $\widetilde{p}_h$ of \eqref{PDEstateLift} is quasi-optimal if $k^2h \ll 1$, that is,
\begin{equation}\label{quasiOptIneq}
\norm{\widetilde{p} - \widetilde{p}_h}_{\mathcal{H},D} \leq C_4(\xi)\norm{\widetilde{p} - \Pi_h\widetilde{p}}_{\mathcal{H},D}.
\end{equation}
\end{prop}

\begin{proof}
Letting $e_h = \widetilde{p} - \widetilde{p}_h$, we define $q\in H^1_{\Gamma_1}(D;\mathbb{C})$ by solving
$$a^*(q,v)=2k^2\int_D e_h\overline{v} - \frac{k\xi_{\rm i}}{|\xi|^2}\int_{\Gamma_2}e_h\overline{v} \quad \forall v\in H^1_{\Gamma_1}(D;\mathbb{C}).$$
Then, $\norm{q}_{\mathcal{H},D} \leq C_1(k,\xi)(2k^2\norm{e_h}_{0,D}+\frac{k|\xi_{\rm i}|}{|\xi|^2}C|e_h|_{1,D})$.
Taking $v=e_h,$ we obtain
\begin{equation}
2k^2\norm{e_h}^2_{0,D}-\frac{k\xi_{\rm i}}{|\xi|^2}\norm{e_h}^2_{0,\Gamma_2} = a^*(q,e_h)=\overline{a(e_h,q)}=\overline{a(e_h,q-\Pi_h q)}.
\end{equation}
It follows that
\begin{equation}\label{bddEdiff}
\begin{aligned}
\norm{e_h}^2_{\mathcal{H},D}  \leq & 2[k^2\norm{e_h}^2_{0,D} + \norm{\nabla e_h}^2_{0,D}] \\
      \leq & 2[\Re a(e_h,e_h)+2k^2\norm{e_h}^2_{0,D} - \frac{k\xi_{\rm i}}{|\xi|^2}\norm{e_h}^2_{0,\Gamma_2}]\\
      \leq & 2[\Re a(e_h, \widetilde{p}-\Pi_h\widetilde{p}) + \overline{a(e_h, q-\Pi_h q)}] \\
      \leq & 2C_0(\xi)\norm{e_h}_{\mathcal{H},D} (\norm{\widetilde{p}-\Pi_h\widetilde{p}}_{\mathcal{H},D} + \norm{q-\Pi_hq}_{\mathcal{H},D}).
\end{aligned}
\end{equation}
Because 
\begin{equation}\label{bddQdiff}
\begin{aligned}
\norm{q-\Pi_hq}_{\mathcal{H},D} =& k\norm{q-\Pi_hq}_{0,D}+ \norm{\nabla(q-\Pi_hq)}_{0,D} \\
	 \leq & C(kh^2+h)|q|_{2,D} \\
	 \leq & C(kh^2+h)C_3(\xi)(k+1)(2k^2\norm{e_h}_{0,D}+\frac{k|\xi_{\rm i}|}{|\xi|^2}\norm{e_h}_{0,\Gamma_2}) \\
	 = & C(\xi)(kh+1)(k^2h+kh)\norm{e_h}_{\mathcal{H},D},
\end{aligned}
\end{equation}
substituting into (\ref{bddQdiff}) into (\ref{bddEdiff}) we obtain
$$\norm{e_h}_{\mathcal{H},D} \leq C(\xi)\norm{\widetilde{p}-\Pi_h\widetilde{p}}_{\mathcal{H},D} + C(\xi)(kh+1)(k^2h+kh)\norm{e_h}_{\mathcal{H},D}.$$
If $k^2h\ll 1$, it is also true that $kh\ll 1$ because $h\ll 1$, so the coefficient in the last term is almost zero. The quasi-optimality (\ref{quasiOptIneq}) then holds.
\end{proof}

From (\ref{H2normbdd}), (\ref{FEappAss}), and (\ref{quasiOptIneq}) we have
\begin{equation*}
\norm{\widetilde{p}-\widetilde{p}_h}_{\mathcal{H},D} \leq C(\xi)(kh^2+h)(k+1)(\Vert\widetilde{f}\Vert_{0,D}+\Vert\widetilde{g}\Vert_{0,\Gamma_{\rm r}}) 
\leq C(\vartheta,\xi)(\Vert p_g\Vert_{0,D}+\Vert p_g\Vert_{0,\Gamma_{\rm r}})h.
\end{equation*}
In the following we would like to provide an error analysis for the POD-based Helmholtz solver. To this end, we assume the finite element approximation property holds in general
\begin{equation}
\norm{\widetilde{p}-\widetilde{p}_h}_{1,D} \leq C(\vartheta, \xi)h \label{lastAssup1}
\end{equation}
with $C(\vartheta,\xi)$ continuous.

Given a sample set $\Xi_{\rm smp}=\{\nu_1,\cdots,\nu_m\}$, we determine snapshots $\{\widetilde{p}_{S,h}^j\}_{j=1}^m$ as mentioned in Section~\ref{pod}, and construct a POD basis $\{\varphi_i\}_{i=1}^N$ of rank $N$ accordingly. Denote the exact solutions corresponding to the snapshots as $\{\widetilde{p}_{S}^j\}_{j=1}^m$. The POD reduced space $\spn\{\varphi_i\}$ is denoted by $V_R$. We make an assumption on the discrete inf-sup condition:
\begin{eqnarray}
&0 < \beta_R(k,\xi) \leq 
\inf\limits_{u\in V_R\backslash0}\sup\limits_{v\in V_R\backslash0}\dfrac{|a_\vartheta(u,v)|}{\norm{u}_{1,D}\norm{v}_{1,D}} \label{lastAssup2}
\end{eqnarray}
where $C(\vartheta,\xi)$ and $\beta_R(k,\xi)$ are continuous functions. We define the $L^2$ projection $\Pi_R: L^2 \to V_R$ for the discussion below, that is $\norm{u-\Pi_R u}_{0,D} = \inf\limits_{v\in V_R}\norm{u-v}_{0,D}$.

\begin{prop}
Under assumption \eqref{lastAssup1} and \eqref{lastAssup2}, the error between the exact solution $\widetilde{p}$ and the reduced solution $\widetilde{p}_R$ is controlled by
\begin{equation}
\norm{\widetilde{p}-\widetilde{p}_R}_{1,D} \leq \Big(1+\frac{C(k,\xi)}{\beta_R(k,\xi)}\Big)\norm{\widetilde{p}-\Pi_R\widetilde{p}}_{1,D}.
\end{equation}
Moreover, the projection error is estimated as
\begin{equation}
\norm{\widetilde{p}-\Pi_R\widetilde{p}}_{1,D}  \leq (1+{\norm{\mathbb{S}}_2^{1/2}})\inf\limits_j\Big(\norm{\widetilde{p}-\widetilde{p}^j_{S}}_{1,D} + C(\nu_j)h\Big) + \Big(\sum\limits_{i=N+1}^d\lambda_i\mathbb{S}_{ii}\Big)^{1/2}
\end{equation}
where $\mathbb{S}$ is the POD stiffness matrix defined as $\mathbb{S}_{jk}=\langle \varphi_k, \varphi_j\rangle + \langle \nabla\varphi_k, \nabla\varphi_j\rangle$, $\norm{\cdot}_2$ is the matrix {\rm2}-norm. Here $\lambda_i$ is an eigenvalue of the snapshot correlation matrix discussed in Section~{\rm\ref{pod}}.
\end{prop}

\begin{proof}
From the discrete inf-sup condition,
\begin{equation} \label{useOrtho}
\begin{aligned}
\norm{\widetilde{p}_R-\Pi_R\widetilde{p}}_{1,D}  \leq & \frac{1}{\beta_R(k,\xi)}\sup\limits_{v\in V_R\backslash 0}\dfrac{|a_\vartheta(\widetilde{p}_R-\Pi_R\widetilde{p}, v)|}{\norm{v}_{1,D}} \\
 = & \frac{1}{\beta_R(k,\xi)}\sup\limits_{v\in V_R\backslash 0}\dfrac{|a_\vartheta(\widetilde{p}-\Pi_R\widetilde{p}, v)|}{\norm{v}_{1,D}}\\
 \leq &\frac{1}{\beta_R(k,\xi)}C(k,\xi)\norm{\widetilde{p}-\Pi_R\widetilde{p}}_{1,D} ,
\end{aligned}
\end{equation}
where second line follows from the fact that $a_\vartheta(\widetilde{p}-\widetilde{p}_R, v)=0~ \forall v\in V_R$ and the last expression from the continuity of $a_\vartheta(\cdot,\cdot)$ similar to (\ref{contiOnHNm}). Therefore, we obtain
\begin{equation}
\norm{\widetilde{p}-\widetilde{p}_R}_{1,D} \leq \norm{\widetilde{p}-\Pi_R\widetilde{p}}_{1,D} + \norm{\widetilde{p}_R-\Pi_R\widetilde{p}}_{1,D} \leq \Big(1+\frac{C(k,\xi)}{\beta_R(k,\xi)}\Big)\norm{\widetilde{p}-\Pi_R\widetilde{p}}_{1,D}.
\end{equation}

We next study the projection error $\norm{\widetilde{p}-\Pi_R\widetilde{p}}_{1,D}$. Recall that the POD basis is determined from the snapshots $\{\widetilde{p}_{S,h}^j\}_{j=1}^m$ computed with samples $\Xi_{\rm smp}=\{\nu_1,\cdots,\nu_m\}$. Denote the exact solutions corresponding to the snapshots as $\{\widetilde{p}_{S}^j\}_{j=1}^m$.
For any $j\in\{1,\cdots,m\}$,
\begin{equation}
\norm{\widetilde{p}-\Pi_R\widetilde{p}}_{1,D}  \leq \norm{\widetilde{p}-\widetilde{p}^j_{S}}_{1,D} + \norm{\widetilde{p}^j_{S}-\widetilde{p}^j_{S,h}}_{1,D} + \norm{\widetilde{p}^j_{S,h}-\Pi_R\widetilde{p}^j_{S,h}}_{1,D} + \norm{\Pi_R\widetilde{p}^j_{S,h}-\Pi_R\widetilde{p}}_{1,D}.
\end{equation}
From \cite[Lemma 2]{Kunisch2001}, the estimate $\norm{v}_{1,D}\leq \sqrt{\vertiii{\mathbb{S}}_2\vertiii{\mathbb{M}^{-1}}_2}\norm{v}_{0,D}$ holds for any $v\in V_R$. Here $\mathbb{S}$ and $\mathbb{M}$ are the POD stiffness and mass matrices, or the Gram matrices of the POD basis associated the $H^1$ and $L^2$ inner product respectively. Note that $\mathbb{M}$ is the identity when taking the $L^2$ norm during POD basis construction. Therefore,
$$
\begin{aligned}
\norm{\Pi_R\widetilde{p}^j_{S,h}-\Pi_R\widetilde{p}}_{1,D}  \leq & \norm{\mathbb{S}}_2^{1/2}\norm{\Pi_R\widetilde{p}^j_{S,h}-\Pi_R\widetilde{p}}_{0,D} \leq  \norm{\mathbb{S}}_2^{1/2}\norm{\widetilde{p}^j_{S,h}-\widetilde{p}}_{0,D} \\
 \leq & \norm{\mathbb{S}}_2^{1/2}\Big(\norm{\widetilde{p}-\widetilde{p}^j_S}_{0,D} + \norm{\widetilde{p}_S^j-\widetilde{p}^j_{S,h}}_{0,D} \Big)\label{11}.
\end{aligned}
$$
From the POD projection error (\ref{H1PODprojErr}) with respect to the $H^1$ norm, we have
$$
\norm{\widetilde{p}^j_{S,h}-\Pi_R\widetilde{p}^j_{S,h}}^2_{1,D} \leq \sum\limits_{i=N+1}^d\lambda_i\norm{\varphi_i}^2_{H^1} = \sum\limits_{i=N+1}^d\lambda_i\mathbb{S}_{ii}.
$$
Combining the above inequalities, we obtain
$$
\norm{\widetilde{p}-\Pi_R\widetilde{p}}_{1,D}  \leq (1+{\vertiii{\mathbb{S}}_2^{1/2}})\inf\limits_j\Big(\norm{\widetilde{p}-\widetilde{p}^j_{S}}_{1,D} + C(\nu_j)h\Big) + \Big(\sum\limits_{i=N+1}^d\lambda_i\mathbb{S}_{ii}\Big)^{1/2}.
$$
\end{proof}

\section{Concluding remarks}
In this work, we pose a stochastic optimization process for the estimation of acoustic liner impedance with the goal of minimizing noise radiation emanating from high-bypass turbofan engines. Uncertainties are introduced into the Helmholtz model to account for variations arising from different weather condition and incomplete knowledge of the fan noise.  We base the optimization on the CVaR measure so that it produces a robust ideal acoustic liner impedance in the presence of uncertainty.

We present a parallel reduced-order modeling framework that dramatically improves the computational efficiency of the stochastic optimization solver on a realistic geometry. Specifically, we build the reduced-order Helmholtz model using 90 POD modes based on only 720 snapshots computed offline. In the Monte Carlo sampling method for approximating the CVaR measure, the computation is parallelized by distributing the MC samples to different processors and solving the corresponding ROMs independently. Whereas a stochastic solution of the full-order Helmholtz solution is forbidding to obtain, the reduced stochastic optimization solver takes less than 500 seconds to execute. Numerical experiments also indicate that an optimal acoustic liner design can control the fan noise radiation, with 95\% certainty, to 48.66\%.

Also provided is mathematical and numerical analyses of the state problem, the optimization problem, and on errors incurred by a finite element discretization. An {\it a posteriori} error analysis for the optimal control problem, as studied in \cite{troltzsch2009pod, gubisch2016posteriori}, is also the interest of the authors. However, this is still open since the control parameter presents in the differential core of the PDE system rather than in the right hand side.

The limitation of the work lies in the lack of an appropriate acoustic liner model that connects the design feature with the impedance factor. This will be a topic of our future work.

\section*{Acknowledgments}
This research work is supported by the US Air Force Office of Scientific Research grant FA9550-15-1-000 and the US Department of Energy grant DE-SC0010678.

\section*{References}

\bibliographystyle{model1b-num-names}
\bibliography{optHelmholtz_3D_v4}

\end{document}